\documentstyle[amssymb,amsfonts,12pt]{amsart}
\newtheorem{theorem}{Theorem}[section]
\newtheorem{lemma}[theorem]{Lemma}
\newtheorem{corollary}[theorem]{Corollary}
\newtheorem{proposition}[theorem]{Proposition}

\newtheorem{conjecture}[theorem]{Conjecture}

\theoremstyle{remark}
\newtheorem{remark}[theorem]{Remark}

\def\QSet{\mbox{\rm\kern.24em
\vrule width.03em height1.48ex depth-.051ex \kern-.26em Q}}

\def\x{{\bf x}}\def\y{{\bf y}}\def\z{{\bf z}}\def\w{{\bf w}}\def\u{{\bf u}}\def\v{{\bf v}}
\def\T{{\mathbb T}}
\def\F{{\mathcal F}}
\def\R{{\mathbb R}}
\def\E{{\mathbb E}}

\def\C{{\mathbb C}}

\def\Z{{\mathbb Z}}
\def\P{{\mathcal P}}
\def\H{{\mathcal H}}
\def\L{{\mathcal L}}\def\A{{\mathcal A}}
\def\det{{\operatorname{det}}}
\def\gcd{{\operatorname{gcd}}}
\def\bas{\begin{align*}}
\def\eas{\end{align*}}
\def\bi{\begin{itemize}}
\def\ei{\end{itemize}}
\newenvironment{proof}{\noindent {\bf Proof} }{\endprf\par}
\def \endprf{\hfill  {\vrule height6pt width6pt depth0pt}\medskip}
\def\emph#1{{\it #1}}

\begin{document}

\title[New bounds for the  discrete Fourier restriction]{New bounds for the  discrete Fourier restriction to the sphere in four and five dimensions}
\author{Jean Bourgain}
\address{School of Mathematics, Institute for Advanced Study, Princeton, NJ 08540}
\email{bourgain@@math.ias.edu}
\author{Ciprian Demeter}
\address{Department of Mathematics, Indiana University, 831 East 3rd St., Bloomington IN 47405}
\email{demeterc@@indiana.edu}

\thanks{The first author is supported by the NSF grant DMS 1301619. The second  author is supported by a Sloan Research Fellowship and by the NSF Grant DMS-1161752}
\begin{abstract}
We improve the  range  for the discrete Fourier restriction to the four and five dimensional spheres. We rely on two new ingredients, incidence theory and  Siegel's mass formula.
\end{abstract}
\maketitle

\section{Introduction}
Let $n\ge 2$ and $\lambda\ge 1$ be two integers. Define $N=[\lambda^{1/2}]+1$ and
$$\F_{n,\lambda}=\{\xi=(\xi_1,\ldots,\xi_n)\in\Z^n:|\xi_1|^2+\ldots|\xi_n|^2=\lambda\}.$$
We will use the notation
$e(z)=e^{iz}.$
Recall the following conjecture from \cite{Bo1}, about the eigenfunctions of the Laplacian on the torus.
\begin{conjecture}
\label{conj1}For each $n\ge 3$, $a_\xi\in \C$, $\epsilon>0$ and each $p\ge \frac{2n}{n-2}$ we have
$$\|\sum_{\xi\in \F_{n,\lambda}}a_\xi e(\xi\cdot x)\|_{L^p(\T^n)}\lesssim_\epsilon N^{\frac{n-2}{2}-\frac{n}{p}+\epsilon}\|a_\xi\|_{l^2(\F_{n,\lambda})}.$$
\end{conjecture}
This can be thought of as a discrete version of the Thomas-Stein restriction theorem. We refer the reader to \cite{Bo0} and \cite{BD} for the necessary background.

Here we make progress when $n=4$ and $n=5$.
\begin{theorem}
\label{mainthm}

(i) For each $a_\xi\in \C$, $\epsilon>0$ and $p>\frac{44}{7}$ we have
$$\|\sum_{\xi\in \F_{4,\lambda}}a_\xi e(\xi\cdot x)\|_{L^p(\T^4)}\lesssim_\epsilon N^{1-\frac{4}{p}+\epsilon}\|a_\xi\|_{l^2(\F_{4,\lambda})}$$

(ii) For each $a_\xi\in \C$, $\epsilon>0$ and $p>\frac{14}{3}$ we have
$$\|\sum_{\xi\in \F_{5,\lambda}}a_\xi e(\xi\cdot x)\|_{L^p(\T^5)}\lesssim_\epsilon N^{\frac{3}{2}-\frac{5}{p}+\epsilon}\|a_\xi\|_{l^2(\F_{5,\lambda})}.$$
\end{theorem}
This improves  the result in \cite{BD} where the conjecture was verified for $p>8$ when $n=4$ and for $p>5$ when $n=5$. The result from \cite{BD} relied essentially on two ingredients. One is the sharp supercritical  estimate ($p>\frac{2(n+1)}{n-3}$) from \cite{Bo1} proved by combining the circle method with the Thomas-Stein argument. See Proposition \ref{earliersupercrit} below.
The second one is the sharp subcritical estimate ($p=\frac{2n}{n-1}$) from \cite{Bo2}, whose proof did not rely at all on number theory.

Here we replace that  subcritical estimate by a new $L^4$ estimate. While this $L^4$ estimate is not sharp, it is strong enough to improve the range in the conjecture. Note that the index $p=4$ is critical when $n=4$ and supercritical when $n\ge 5$. Thus, when $n=4$, $L^4$ is precisely the right space to consider; getting the sharp $L^4$ estimate would completely solve Conjecture \ref{conj1}.
On the other hand, the $L^4$ approach we develop is only useful for $n=4,5$ since the result in \cite{BD} already proved the sharp $L^4$ bound in dimensions $n\ge 6$.

To derive the $L^4$ estimate we rely on two new methods: incidence theory and Siegel's mass formula.
Interestingly, the application of both methods is rather sharp, see Remarks \ref{remarkmayb1}, \ref{remarkmayb2}, \ref{jdcvyr7uyr7f67r4} and \ref{jdcvyr7uyr7f67r4kk}.
We mention that the use of incidence theory, while new in the context of Conjecture \ref{conj1}, has been in the last twenty years or so one of the important tools in various other problems with restriction theory flavor. It suffices to mention \cite{TWol} and the more recent \cite{Lew}, \cite{BoBo}.

We describe the incidence theory approach in sections \ref{sec2} and \ref{sec3} while the number theoretical approach appears in sections \ref{sec4} and \ref{sec5}. These tools are then combined in section \ref{sec6} to prove our main theorem. In the last section we speculate on possible ways to further improve our result.

The first author would like to thank Peter Sarnak for clarifying
discussions around the Siegel mass formula.

\section{Some background from incidence theory}
\label{sec2}

Let $\P$ be a collection of points in $\R^n$ and let $\H$ be a collection of sets in $\R^n$. We will not assume at this point that $\H$ consists of hyperplanes.  Consider the standard incidence bipartite graph $G(\P,\H)$ with vertex sets $\P$ and $\H$, where we have  an edge between $P\in\P$ and $H\in\H$ whenever we have the incidence $P\in H$. So the number of edges $E$ in $G(\P,\H)$ is the same as the number of incidences $I(\P,\H)$ between $\P$ and $\H$.

Our approach in this section is an adaptation of Theorem 8 from \cite{ES} to our needs. The next two lemmas prove some weaker bounds that are then amplified to optimal bounds in Proposition \ref{Propoptinc}. The first one is tailored for applications to four dimensions, while the second one for five dimensions.
\begin{lemma}
\label{lematrivinc1}
Fix $\gamma>4$.  Assume $|H\cap H'\cap \P|\le \gamma$ for each $H\not= H'\in \H$. Then
$$I(\P,\H)\le \gamma(\;|\P|+|\H|\sqrt{|\P|}\;)$$
\end{lemma}
\begin{proof}
The argument is a standard double counting. Let ${\P}_H$ be the points in ${\P}\cap H$ and let ${\H}_P$ be the sets in $\H$ that contain the point $P$. We estimate $\sum_{H\in \H}I({\P}_H,\H)$ from above by
$$\le \sum_{H\in \H}(\gamma|{\H}|+|{\P}_H|)=\gamma|{\H}|^2+|E|$$
and also from below by
$$\sum_{P\in\P}|{\H}_P|^2\ge \frac1{|\P|}(\sum_{P\in\P}|{ \H}_P|)^2=\frac{E^2}{|\P|}.$$
Thus
$$|{\P}|\gamma|{\H}|^2\ge |E|^2( 1-\frac{|{\P}|}E).$$
Then either $E\le 2|\P|$ or, if not, the above implies $E^2\le 2|{\P}|\gamma|{\H}|^2$. In either case we are fine.

\end{proof}

\begin{lemma}
\label{lematrivinc2}
Fix $\gamma>4$.  Assume that for each $H\not= H'\in \H$
$$|\{H''\in\H:|H\cap H'\cap H''\cap \P|\ge \gamma\}|\le \gamma.$$
Then
$$I(\P,\H)\le 4\gamma(\;|\P|+|\H||\P|^{2/3}\;)$$
\end{lemma}
\begin{proof}
We will apply twice the double counting argument from the proof of Lemma \ref {lematrivinc1}. Let as before ${\P}_H$ be the points in ${\P}\cap H$ and let ${\H}_P$ be the sets in $\H$ that contain the point $P$. Set $E_1=I(\P,\H)$. Let  us assume for a moment that we have the following inequality for each $H\in\H$
\begin{equation}
\label{equtzionn1}
I(\P_H,\H)\le 2\gamma(|\P_H|+|\H|\sqrt{|\P_H|})
\end{equation}
We can then estimate $\sum_{H\in \H}I({\P}_H,\H)$ from above using Cauchy-Schwartz by
$$\sum_{H\in \H}2\gamma(|\P_H|+|\H|\sqrt{|\P_H|})=2\gamma(E_1+|\H|^{3/2}E_1^{1/2})$$
and also from below by
$$=\sum_{P\in\P}|{\H}_P|^2\ge \frac1{|\P|}(\sum_{P\in\P}|{ \H}_P|)^2=\frac{E_1^2}{|\P|}.$$
Thus
$$2\gamma|{\P}||{\H}|^{3/2}\ge E_1^{3/2}( 1-\frac{2\gamma|{\P}|}{E_1}).$$
Then either $E_1\le 4\gamma|\P|$ or, if not, the above implies $E_1^{3/2}\le 4\gamma|{\P}||{\H}|^{3/2}$. In either case we are fine.

It remains to prove \eqref{equtzionn1}. Fix $H\in\H$ and define for simplicity $\P'=\P_H$, $H'=\H\setminus \{H\}$ and $E=I(\P',\H')$. Since $I(\P',\H)=I(\P',\H')+|\P'|$, we are left with estimating $I(\P',\H')$. We apply again the double counting argument. Let as before ${\P'}_{H'}$ be the points in ${\P'}\cap H'$ and let ${\H'}_P$ be the sets in $\H'$ that contain the point $P$. For each $H'\in \H'$ define $\H'_{1,H'}$ to consist of those $H''\in \H'$ such that $|H'\cap H''\cap \P'|\ge \gamma$. Note that our hypothesis implies that $|\H'_{1,H'}|\le\gamma$. Let $\H'_{2,H'}=\H'\setminus \H'_{1,H'}$.

 We estimate $\sum_{H'\in \H'}I({\P'}_{H'},\H')$ from above by
$$\sum_{H'\in \H'}I({\P'}_{H'},\H'_{1,H'})+\sum_{H'\in \H'}I({\P'}_{H'},\H'_{2,H'})\le \sum_{H'\in \H'}(\gamma|{\P'_{H'}}|+\gamma|{\H'}|)=\gamma(E+|{\H'}|^2)$$
and also from below by
$$=\sum_{P\in\P'}|{\H'}_P|^2\ge \frac1{|\P'|}(\sum_{P\in\P'}|{ \H'}_P|)^2=\frac{E^2}{|\P'|}.$$
Thus
$$|{\P'}|\gamma|{\H'}|^2\ge |E|^2( 1-\frac{\gamma|\P'|}E).$$
Then either $E\le 2|\P'|$ or, if not, the above implies $E^2\le 2|{\P'}|\gamma|{\H'}|^2$. In either case we are fine.

\end{proof}

From this point on we assume $\H$ consists of hyperplanes.

\begin{proposition}
\label{Propoptinc}
Let $\H$ of finite collection of hyperplanes in $\R^n$   and let $\P$ be a finite collection of points  in $\R^n$. Assume the following hold for a given $\gamma\ge 1$

(a) $I(\P',\H')\le \gamma(\;|\P'|+|\H'||\P'|^{\frac{n-3}{n-2}}\;)$ for each $\mathcal P'\subset \mathcal P$,
$\mathcal H'\subset \mathcal H$

(b) Any $\gamma$ hyperplanes in $\H$ share fewer than $\gamma$ points in $\P$

Then the number of incidences satisfies for each $\epsilon>0$
\begin{equation}
\label{numberofincid}
I({\P},{\mathcal H})\le C_\epsilon\gamma(|{\P}|^{\alpha}|{\H}|^{\beta}+|{\P}|+|{\H}|(1+\log_2|{\P}|)),
\end{equation}
where $\alpha=\frac{n(n-3)}{n^2-2n-1}$, $\beta=\frac{(n-1)(n-2)}{n^2-2n-1}+\epsilon$ and $C_\epsilon$ depends only on $\epsilon$ and $n$.
\end{proposition}

Recall the following Cutting Lemma (see Theorem 6.5.3 in \cite{Ma}). This will enable a proof of the Proposition via induction.
\begin{lemma}
\label{Cutting_Lemma}
Given $s$ hyperplanes in $\R^n$ and a positive integer $r<s$, there exists a partition of $\R^n$ into fewer than $r^n$ parts, such that for each part there are at most $Bs/r$ hyperplanes which cut it (this means intersect it without containing it). $B$ will be a large number depending on $n$, but independent of $s,r$.
\end{lemma}

We now begin the proof of the Proposition \ref{Propoptinc}, following \cite{ES}. By performing a translation, we can assume that neither of the hyperplanes in $\H$ contains the origin ${\bf 0}$,  and also that ${\bf 0}\notin \mathcal P$.

Choose $r$ large enough so that
\begin{equation}
\label{13-e1}
\frac{B^\alpha}{r^{\alpha-n(1-\beta)}}+\frac{B}{r}<\frac12.
\end{equation}
Here $B$ is the constant from Lemma \ref{Cutting_Lemma}. Note that $\alpha>n(1-\beta)$ and also that $r$ will only depend on $n,\epsilon$.

Let $s=|\P|$ and $t=|\H|$. We prove \eqref{numberofincid} via induction on $s$. The case $s=1$ holds trivially true.  Assume now it holds for $1,\ldots,s-1$. We split the analysis in three cases.

Case1: If $s\le r$ then we trivially have $I(s,t)\le st\le sr$, so it suffices to choose $C_\epsilon>r$

Case 2: If $r^{\frac n{1-\alpha}}s\ge t^{n-2}$ then hypothesis (a) implies that
$$I(s,t)\le \gamma(s+ts^{\frac{n-3}{n-2}})\le \gamma s(1+r^{\frac{n}{(1-\alpha)(n-2)}}),$$
so it suffices to choose $C_\epsilon>1+r^{\frac{n}{(1-\alpha)(n-2)}}$.

Case 3: We now focus on the case when $r<s$ and $r^{\frac n{1-\alpha}}s< t^{n-2}$. By raising both terms in the second inequality to the power $1-\alpha$ we get
\begin{equation}
\label{13-e2}
s<r^{-n}s^\alpha t^\beta.
\end{equation}
Next, we dualize. That is, we
identify each point  $P\in \P$ with the hyperplane
$$H_P:=\{x\in\R^n:\langle x, P\rangle=1\}$$
(call the resulting collection $\H'$) and each hyperplane $H\in \H$ with the point $P\in \R^n\setminus\{\textbf{0}\}$ such that
$$H_P=H$$
(call the resulting collection $\P'$). It is easy to see that incidences are preserved, that is $P\in H_{P'}$ if and only if $P'\in H_P$.

Apply the cutting lemma to the collections $\H'$, $\P'$ and $r$. Note that we operate under the assumption $r<s$, which makes the lemma applicable. Assign each point in $\P'$ to the part that contains it, and to each part we assign all hyperplanes in $\H'$ which cut it . So a hyperplane can be assigned to more than one part, and there may be parts that are not assigned any hyperplanes. Call $s_i$ and $t_i$ the hyperplanes and points assigned to the $i^{th}$ of the $M$ parts. We have
$$M\le r^n,\,\,\sum_{i=1}^Mt_i=t,\,\,s_i\le \frac{Bs}{r}.$$
Each part contributes with two types of incidences. First, with the hyperplanes that cut it. Second, with those that contain it.
The first contribution is bounded using the induction hypothesis (after undualizing) by
$$I(t_i,s_i)\le C_\epsilon\gamma(s_i^{\alpha}t_i^{\beta}+s_i+t_i(1+\log_2 s_i)).$$
The second contribution is bounded by
$\gamma(s+t)$.
Indeed, if the part contains fewer than $\gamma$ points then there are at most $\gamma s$ incidences. If there are at least $\gamma$ points in the part, there can be at most $\gamma$ hyperplanes in $\H'$ containing the part (undualize and use hypothesis (b)).  Thus, there are fewer than $\gamma t$ incidences. We conclude that
$$I(s,t)\le  C_\epsilon\gamma\sum_{i=1}^M (s_i^{\alpha}t_i^{\beta}+s_i+t_i(1+\log_2 s_i))+r^n\gamma(s+t)$$
$$\le C_\epsilon\gamma[(\frac{Bs}{r})^\alpha\sum_{i=1}^Mt_i^\beta+\frac{BMs}{r}+t(1+\log_2\frac{Bs}{r})]+r^n\gamma(s+ t).$$
Using \eqref{13-e1}  we can further bound this by
$$C_\epsilon\gamma[(\frac{Bs}{r})^\alpha M^{1-\beta}(\sum_{i=1}^Mt_i)^\beta+Bsr^{n-1}+s\frac{r^n}{C_\epsilon}+t(\frac{r^n}{C_\epsilon}+\log_2 s)].$$
Since the second term does not fit well ($Br^{n-1}$ is greater than 1), we need to replace it using \eqref{13-e2}. We further bound the above  by
$$C_\epsilon\gamma [(\frac{B^\alpha}{r^{\alpha-n(1-\beta)}}+\frac{B}{r})s^\alpha t^\beta+s\frac{r^n}{C_\epsilon}+t(\frac{r^n}{C_\epsilon}+\log_2 s)].$$
It now suffices to choose $C_\epsilon>r^n$. This ends the proof of the Proposition.

\section{The incidence theory approach}
\label{sec3}

For $\Lambda\subset \R^n$ define its additive energy
$$\E(\Lambda)=|\{(\xi_1,\xi_2,\xi_3,\xi_4)\in \Lambda^4:\;\xi_1+\xi_2=\xi_3+\xi_4\}|.$$
We now show how to use the incidence theory developed so far to estimate the additive energy of subsets of the sphere.
We will rely on the well known estimates, see \cite{Gr}
\begin{equation}
\label{sharpestnumberlatt234}
|\F_{n,\lambda}|\lesssim_\epsilon N^{n-2+\epsilon},\;\;n=2,3,4
\end{equation}
\begin{equation}
\label{sharpestnumberlatt}
|\F_{n,\lambda}|\approx N^{n-2},\;\;n\ge 5
\end{equation}

For $v\in\Z^n$ with $|v|<2\lambda^{1/2}$ let $H_v$ be the unique hyperplane in $\R^n$ containing the $n-2$ dimensional sphere
$$S_v=\{\xi\in \lambda^{1/2}S^{n-1}:\|\xi-v\|=\lambda^{1/2}\}.$$
\begin{theorem}
\label{t_incfour}
Let $\Lambda$ be an arbitrary subset of $\F_{4,\lambda}$. Then its energy satisfies for each $\epsilon>0$
\begin{equation}
\label{partialenergestfourdim}
\E(\Lambda)\lesssim_\epsilon N^{\epsilon}|\Lambda|^{7/3}.
\end{equation}
\end{theorem}
\begin{proof}
Note that given $\xi,\eta\in \lambda^{1/2}S^3$, we have that $\xi+\eta=v$ if and only if $\xi,\eta\in S_v$ and $\xi,\eta$ are diametrically opposite on $S_v$. For $0\le k\le [\log_2 \Lambda]$ let $M_k$ denote the number of hyperplanes $H_v$ containing between $2^k$ and $2^{k+1}-1$ pairs $(\xi,\eta)\in \Lambda^2$ such that $\xi+\eta=v$. Since $\E(\Lambda)\le \sum_{k}2^{2k+2}M_k$ and $|\Lambda|\lesssim N^2$, it suffices to prove that for each $k$
\begin{equation}
\label{whatneedsviasum}
M_k2^{2k}\lesssim_\epsilon N^{\epsilon}|\Lambda|^{7/3}.
\end{equation}
We will find two upper bounds for $M_k$. First, note the trivial bound
\begin{equation}
\label{first_triv_dya}
M_k2^{k}\lesssim |\Lambda|^2.
\end{equation}
Next, note that $M_k$ is smaller that the number $N_k$ of hyperplanes $H_v$ -call the collection $\H$- satisfying $$2^k\le |H_v\cap \Lambda|.$$
Recall  that any circle on $\lambda^{1/2}S^3$ contains  $O(N^\epsilon)$ points in $\Z^4$, see \cite{BP}. Thus $\H$ satisfies the requirement in Lemma \ref{lematrivinc1}, for $\gamma$ large enough but satisfying $\gamma\lesssim_\epsilon N^\epsilon$ for each $\epsilon>0$. Note that there are at least $2^kN_k$ incidences between $\H$ and $\Lambda$. Apply now Proposition \ref{Propoptinc} with $\P=\Lambda$ and $n=4$ to get
\begin{equation}
\label{fjerhugiretut8gtutg}
N_k2^k\lesssim_\epsilon N^{\epsilon}(|\Lambda|^{\frac47}N_k^{\frac67+\epsilon}+|\Lambda|+N_k(1+\log_2|\Lambda|)).
\end{equation}
If $N_k2^k\lesssim_\epsilon N^{\epsilon}|\Lambda|^{\frac47}N_k^{\frac67+\epsilon}$, then since $N_k\lesssim N^4$ we get
$$N_k\lesssim_\epsilon N^{\epsilon}\frac{|\Lambda|^{4}}{2^{7k}}.$$
Combining this with \eqref{first_triv_dya} and $M_k\le N_k$ gives \eqref{whatneedsviasum}.

If either $N_k2^k\lesssim_\epsilon N^{\epsilon}|\Lambda|$ or  $N_k2^k\lesssim_\epsilon N^{\epsilon}N_k(1+\log_2|\Lambda|)$ then \eqref{whatneedsviasum} follows immediately from \eqref{first_triv_dya}.

\end{proof}

\begin{remark}
\label{remarkmayb1}
Note that the expected result is
$$\E(\Lambda)\lesssim_\epsilon N^\epsilon |\Lambda|^2.$$
To prove \eqref{partialenergestfourdim} we have relied on the incidence bound \eqref{fjerhugiretut8gtutg}. This bound holds for any collection of hyperplanes in $\R^4$ subject to the only requirement that any two of them share at most $O(N^\epsilon)$ points in $\Lambda$. We now show that \eqref{fjerhugiretut8gtutg} can not in general  be improved unless this requirement is strengthened in some way.

On the one hand, note that the argument in the proof of Theorem \ref{t_incfour} can be applied with no essential modifications to the paraboloid
$$P^3_N=\{\xi:=(\xi_1,\xi_2,\xi_3,\xi_1^2+\xi_2^2+\xi_3^2):|\xi_i|\le N\}.$$
Indeed, if $\xi+\eta=v:=(v_1,v_2,v_3,v_4),$ then
$$\sum_{i=1}^3[\xi_i^2+(\xi_i-v_i)^2]=v_4,$$
and thus $\xi,\eta$ belong to the hyperplane
$$H_v:=\{\theta\in\R^4:\;2v_1\theta_1+2v_2\theta_2+2v_3\theta_3-2\theta_4=v_1^2+v_2^2+v_3^2-v_4\}.$$
Next note that for $v\not= v'$ the projection onto the first three components of  $H_v\cap H_{v'}\cap P^3_N$ is a subset of $C\cap \Z^3$, where $C$ is  a certain circle of radius $O(N)$. Thus  $$|H_v\cap H_{v'}\cap P^3_N|\lesssim_\epsilon N^\epsilon.$$

On the other hand, the estimate \eqref{partialenergestfourdim} is sharp for $\Lambda=P^3_N\cap \Z^4$. Indeed, note that
$\E(P_N^3)=\|K\|_{L^4(\T^4)}^4$, where
$$K(x)=\sum_{\xi\in P_N^3}e(\xi\cdot x).$$
Since $|\xi\cdot x|\ll 1$ for $|x_1|,|x_2|,|x_3|\ll\frac1N$ and $|x_4|\ll\frac1{N^2}$, it follows that $|K(x)|\gtrsim N^3$ for $x$ in a set of measure $\gtrsim \frac1{N^5}$. This shows $\|K\|_4^4\gtrsim N^7$.
\end{remark}
\medskip

We will now obtain a similar result in five dimensions. The new observation that we need in this case is
\begin{lemma}
\label{saviourlemma}
There are $O(N^\epsilon)$ hyperplanes $H_v$ in $\R^5$ containing a given three dimensional affine subspace $W$ of $\R^5$.
\end{lemma}
\begin{proof}
This will follow from a few easy observations. Call $V$ the collection of all such $v$ and fix $\eta\in W$.
 First, it is easy to see that $v$ is orthogonal to $H_v$, in particular each $v\in V$ is orthogonal to $W$.  Second, note that if $\xi\in S_v$ then also $v-\xi\in S_v$. Thus $\xi,v-\xi\in H_v$ which forces $v/2\in H_v$. Combining this with the first observation  further implies that $\langle v/2,\eta-v/2\rangle=0$. Thus $v/2$ belongs to the sphere centered at $\eta/2$ of radius $|\eta|/2$. Since $v/2$ is orthogonal to $W$, it is confined to a two dimensional subspace. As a result, all $v\in V$ will belong to a fixed circle of radius $O(\lambda^{1/2})$. It now suffices to invoke again the result in \cite{BP}.
 \end{proof}
 \begin{theorem}
\label{t_incfive}
Let $\Lambda$ be an arbitrary subset of $\F_{5,\lambda}$. Then its energy satisfies for each $\epsilon>0$
\begin{equation}
\label{dbfhrgyfrkwpp-qp=-}
\E(\Lambda)\lesssim_\epsilon N^{\epsilon}|\Lambda|^{5/2}.
\end{equation}
\end{theorem}
\begin{proof}
The analysis is very similar with that in the proof of Theorem \ref{t_incfour}, we will use the notation $M_k,N_k,\H$ from there. The crucial difference is that the new collection $\H$ does not satisfy the requirement in Lemma \ref{lematrivinc1}, since two hyperplanes intersect $\lambda^{1/2}S^4$ along a three dimensional sphere that may contain as many as $N$ points in $\Lambda$. However, Lemma \ref{saviourlemma} shows that Lemma \ref{lematrivinc2} is applicable in our situation. We will choose as before a $\gamma$ large enough but satisfying $\gamma\lesssim_\epsilon N^\epsilon$ for each $\epsilon>0$. Indeed, given distinct $H,H',H''\in \H$ with $|H\cap H'\cap H''\cap \Lambda|\ge \gamma$, it must be that $H\cap H'\cap H''$ is a three dimensional linear subspace $W$. This is because any lower dimensional subspace contains fewer that $\gamma$ points, if $\gamma$ is chosen large enough. But then $W=H\cap H'$ and  thus $H''$ contains $W$. Lemma \ref{saviourlemma} produces the desired upper bound.

Apply Proposition \ref{Propoptinc} to $\H$, $\P=\Lambda$ and $n=5$ to get
$$N_k2^k\lesssim_\epsilon N^{\epsilon}(|\Lambda|^{\frac57}N_k^{\frac67+\epsilon}+|\Lambda|+N_k(1+\log_2|\Lambda|)).$$
The argument then follows closely the lines of that in the proof of Theorem \ref{t_incfour}.
\end{proof}
\begin{remark}
\label{remarkmayb2}
The analogue of Remark \ref{remarkmayb1} applies in this context, too.
\end{remark}

\section{Counting solutions to systems of quadratic equations}
\label{sec4}

In this section we develop the necessary number theoretical machinery that will enable us to prove a different type of estimate for the energy of the lattice points on the sphere. The main theorem is as follows.

\begin{theorem}
\label{nSiegel}

(a) Let  $N_{a,b,\lambda}$ be the number of solutions $(\x,\y,\z)\in(\Z^4)^3$ of the system of equations $$\begin{bmatrix}x_1&x_2&x_3&x_4\\y_1&y_2&y_3&y_4\\z_1&z_2&z_3&z_4\end{bmatrix}\begin{bmatrix}x_1&y_1&
z_1\\x_2&y_2&
z_2\\x_3&y_3&
z_3\\x_4&y_4&
z_4\end{bmatrix}=\begin{bmatrix}\lambda&a&\lambda+a-b\\a&\lambda&b\\ \lambda+a-b&b&\lambda\end{bmatrix}$$
Then $$\sum_{|a|,|b|\le \lambda}N_{a,b,\lambda}\lesssim_\epsilon \lambda^{2+\epsilon}$$ for each $\epsilon>0$.

(b)  Let  $N_{a,b,c,d,\lambda}$ be the number of solutions $(\u,\v,\x,\y)\in(\Z^5)^4$ with $\x\not=\y$ of the system of equations $$\begin{bmatrix}u_1&u_2&u_3&u_4&u_5\\v_1&v_2&v_3&v_4&v_5\\x_1&x_2&x_3&x_4&x_5\\y_1&y_2&y_3&y_4&y_5
\end{bmatrix}\begin{bmatrix}u_1&v_1&
x_1&y_1\\u_2&v_2&
x_2&y_2\\u_3&v_3&
x_3&y_3\\u_4&v_4&
x_4&y_4\\u_5&v_5&
x_5&y_5\end{bmatrix}=\begin{bmatrix}a&c&
a/2&a/2\\c&b&
b/2&b/2\\a/2&b/2&
\lambda&d\\a/2&b/2&
d&\lambda\end{bmatrix}$$
Then $$\sum_{|a|,|b|,|c|,|d|\lesssim \lambda}N_{a,b,c,d,\lambda} \lesssim_\epsilon \lambda^{4+\epsilon}$$ for each $\epsilon>0$.
\end{theorem}

Note that for part (b) we have to exclude solutions with $\x=\y$. The computations from Section \ref{sec5} show that the sum over $|a|,|b|,|c|,|d|\lesssim \lambda$ of this type of solutions is roughly
$$|\{(\u,\v,\x)\in\Z^5\times\Z^5\times \F_{5,\lambda}:\u-\x,\v-\x\in\F_{5,\lambda} \}|=|\F_{5,\lambda}|^3\sim\lambda^{9/2}.$$

Our main tool will be Siegel's mass formula which we recall below. In a nutshell, this formula relates the number of integral solutions to a system of quadratic equations with the number of solutions of the same system in $\Z_{p^r}$, with $p$  prime and $r\to\infty$. The necessary background and the proof of Siegel's mass formula are in \cite{Si}.
More precisely, we will use the formula on page 10, case (i) from Lecture No. 2, which is  proved in Lecture No. 6.

Let $m\ge n+1$ and let $\gamma\in M_{m,m}(\Z)$ and $\Lambda\in M_{n,n}(\Z)$ be two positive definite matrices with integer entries. Denote by $A(\gamma,\Lambda)$  the number of  solutions $\L\in M_{m,n}(\Z)$ for
\begin{equation}
\label{Snew27}
\L^*\gamma\L=\Lambda.
\end{equation}
Then Siegel's mass formula asserts that
\begin{equation}
\label{Snew1}
[\sum_{i=1}^h\frac{A(\gamma_i,\Lambda)}{A(\gamma_i,\gamma_i)}][\sum_{i=1}^h\frac1{A(\gamma_i,\gamma_i)}]^{-1}=
C_{n,m,\gamma}A_0(\gamma,\Lambda)\prod_{p\text{ prime}}\nu_p(\gamma,\Lambda).
\end{equation}
Here  $h$ is the number of classes in the genus of $\gamma$ and $\gamma_i$ is a (any) representative for its class. On pages 9 and 10 of \cite{Si} it is stated that
$$\nu_p(\gamma,\Lambda)=\lim_{r\to\infty}\frac1{p^{r(mn-\frac{n(n+1)}{2})}}|\{\L\in M_{m,n}(\Z_{p^r}):\;\L^*\L\equiv \Lambda \mod p^r\}|,$$
while the computations on page 41 in \cite{Si} show that
\begin{equation}
\label{Snew21}
A_0(\gamma,\Lambda)=K_{n,m}(\det(\gamma))^{-n/2}(\det(\Lambda))^{\frac{m-n-1}{2}}.\end{equation}
As $A(\gamma,\Lambda)=A(\gamma_i,\Lambda)$ for some $1\le i\le h$, we immediately get
\begin{equation}
\label{Snew1863}
A(\gamma,\Lambda)\lesssim_{n,m,\gamma} (\det(\Lambda))^{\frac{m-n-1}{2}}\prod_{p\text{ prime}}\nu_p(\gamma,\Lambda).
\end{equation}
In our forthcoming applications  $m=n+1$, $\gamma$ will always  be the identity matrix $I_{n+1}$ and $m$ is the dimension of the ambient space where the lattice points live. 
\bigskip

Fix $\Lambda\in M_{n,n}(\Z)$, a nonsingular positive definite matrix, in particular $\det(\Lambda)\not=0$. In evaluating $\nu_p(I_{n+1},\Lambda)$ we distinguish two separate cases: $p\nmid\det(\Lambda)$ and $p|\det(\Lambda)$. We start with the first case.

\begin{proposition}
\label{propnondiv}
Assume $p$ is not a factor of $\det (\Lambda)$. Then
$$\nu_p(I_{n+1},\Lambda)\le 1+\frac{C}{p^2},$$
where $C$ is independent of $p,\Lambda.$
\end{proposition}
To prove the proposition we first analyze the case $r=1$. Using the same invariance considerations as in the evaluation of the term $A_0$ in \cite {Si} Lecture 6, we get that
\begin{equation}
\label{Snew2}
|\{\L\in M_{n+1,n}(\Z_{p}):\;\L^*\L\equiv \Lambda \mod p\}|
\end{equation}
only depends on the Legendre symbol $(\frac{\det(\Lambda)}{p})$. Thus we can replace $\Lambda$ with the diagonal matrix
\begin{equation}
\label{Snew3}
\Lambda_\xi=\xi e_1\otimes e_1+\sum_{j=2}^{n} e_j\otimes e_j
\end{equation}
where $(\frac{\xi}{p})=(\frac{\det(\Lambda)}{p})$.

We will rely on the following elementary fact, see Exercise 13 on page 31 in \cite{Cass}
\begin{lemma}
Let $g(\x)=[\x]^{*}C[\x]$ be a quadratic form with $C\in M_{l,l}({\mathbb F}_p)$ symmetric and $d:=\det(C)\not\equiv 0\mod p$. Denote for $\xi\in{\mathbb F}_p$
$$N_\xi(d,l)=|\{\x\in({\mathbb F}_p)^l:g(\x)=\xi\}|.$$
Then
\begin{equation}
\label{Snew4}
N_0(d,l)=\begin{cases}p^{l-1}&:\quad \text{if }l=2\nu+1,\;\nu\ge 0 \\ \hfill  p^{l-1}+(p-1)p^{\nu-1}\left(\frac{(-1)^{\nu}d}{p}\right)&:\quad \text{if }l=2\nu,\;\nu>0 \end{cases},
\end{equation}
\begin{equation}
\label{Snew5}
N_1(d,l)=\begin{cases}p^{l-1}+p^{\nu}\left(\frac{(-1)^{\nu}d}{p}\right)&:\quad \text{if }l=2\nu+1,\;\nu\ge 0 \\ \hfill  p^{l-1}-p^{\nu-1}\left(\frac{(-1)^{\nu}d}{p}\right)&:\quad \text{if }l=2\nu,\;\nu>0 \end{cases},
\end{equation}
where $(\frac {a}p)$ is the Legendre symbol.
\end{lemma}
It is easy to see that if $\eta\in {\mathbb F}_p^*$ then $N_\xi(d,l)$ only depends on the class of $\eta$ in the two element group ${\mathbb F}_p^*/({\mathbb F}_p^*)^2$. Thus
$$N_0(d,l)+\frac{p-1}{2}N_1(d,l)+\frac{p-1}{2}N_\eta(d,l)=p^l$$
for each $\eta\in{\mathbb F}_p^*\setminus({\mathbb F}_p^*)^2$, and we conclude that
\begin{equation}
\label{Snew6}
N_\eta(d,l)=\begin{cases}p^{l-1}-p^{\nu}\left(\frac{(-1)^{\nu}d}{p}\right)&:\quad \text{if }l=2\nu+1,\;\nu\ge 0 \\ \hfill  p^{l-1}-p^{\nu-1}\left(\frac{(-1)^{\nu}d}{p}\right)&:\quad \text{if }l=2\nu,\;\nu>0 \end{cases}.
\end{equation}

We now evaluate
\begin{equation}
\label{Snew7}
|\{\L\in M_{n+1,n}(\Z_{p}):\;\L^*\L\equiv \xi e_1\otimes e_1+\sum_{j=2}^{n} e_j\otimes e_j\mod p\}|
\end{equation}
for $n\ge 2$. We need to count pairwise orthogonal vectors $\x^1,\ldots,\x^{m-1}\in{\mathbb F}_p^{n+1}$ such that
$$\x^1\cdot \x^1=\xi\text{  and  }\x^j\cdot\x^j=1 \text{ for }2\le j\le n.$$
We have $N_\xi(1,n+1)$ choices for $\x^1$. Once we have chosen $\x^1$, there will be $N_1(1,n)$ possibilities for $\x^2$, since $\x^2\in (\x^1)^{\perp}$. By repeating this reasoning and then using \eqref{Snew5}, \eqref{Snew6}, we bound the term \eqref{Snew7} by
$$N_\xi(1,n+1)\prod_{k=2}^{n}N_1(1,k)\le$$
\begin{equation}
\label{jewur8t75865iogkjjbhwyt56231830eiooerjgiotvljiogu}
p^{n}(1+\frac{1}{p^2})\ldots p^{3}(1+\frac{1}{p^2})(p^2+p(\frac{-1}{p}))(p-(\frac{-1}{p}))=p^{\frac{n(n+1)}{2}}(1+O(\frac{1}{p^2})).
\end{equation}

To get the proof of Proposition \ref{propnondiv} we need to recall Hensel's lemma (see for example \cite{Gree}, Chapter 5 and \cite{Wo})
\begin{lemma}
\label{lemma:Hensel}
Let $r\ge 1$.
Let $f_1,\ldots,f_k\in\Z[X_1,\ldots,X_k]$ for $1\le j\le k$ be a collection of polynomials, and set
$$J({\bf X})=\det (\frac{\partial f_j}{\partial X_i}({\bf X}))_{1\le i,j\le k}.$$
Let ${\bf X}\in\Z_p^k$ be a solution of 
$$f_j({\bf X})\equiv 0\mod p,\;\;1\le j\le k$$
such that  $J({\bf X})\not\equiv 0 \mod p$. Then there is a unique solution ${\bf X'}\in\Z_{p^r}^k$ of 
$$f_j({\bf X'})\equiv 0\mod p^r,\;\;1\le j\le k$$ that satisfies ${\bf X'}\equiv {\bf X}\mod p$.
\end{lemma}

Recall that we need to count the number of solutions for
\begin{equation}
\label{Snew9}
\L\in M_{n+1,n}(\Z_{p^r}):\;\L^*\L\equiv \Lambda \mod p^r
\end{equation}
Since $\det(\Lambda)\not\equiv 0\mod p$, it follows that any solution $\L=(\x^1_0,\ldots,\x^n_0)$ consists of linearly independent  vectors $\x^i_0$ over $\Z_p$. Equivalently, the rank $\mod p$ of the matrix $\L$ is maximal (it equals $n$). This implies that the rank of
the $\frac{n(n+1)}{2}\times n(n+1)$ matrix $M(\x^1_0,\ldots,\x^n_0)$ whose entries are the partial derivatives (evaluated at the point $\L$) of the functions $\x^i\cdot\x^j$,  with respect to the variables $x_1^1,\ldots,x_n^{n+1}$ is maximal (it equals $\frac{n(n+1)}{2}$). Hensel's lemma with $k=\frac{n(n+1)}{2}$ shows that each solution for $$\L\in M_{n+1,n}(\Z_{p}):\;\L^*\L\equiv \Lambda \mod p$$
gives rise to exactly $p^{\frac{n(n+1)}{2}(r-1)}$ solutions for
$$\L'\in M_{n+1,n}(\Z_{p^r}):\;(\L')^*\L'\equiv \Lambda \mod p^r$$
such that $\L'\equiv \L\mod p$. Indeed, for each such $\L$, let $X_1,\ldots,X_{\frac{n(n+1)}{2}}$ be the  variables among $x_1^1,\ldots,x_{n+1}^{n}$ that correspond to $\frac{n(n+1)}{2}$ independent columns of $M(\x^1_0,\ldots,\x^n_0)$. Also let $f_1,\ldots,f_{\frac{n(n+1)}{2}}$ be the functions $\x^i\cdot\x^j$, considered as functions of only the variables $X_1,\ldots,X_{\frac{n(n+1)}{2}}$. The remaining $\frac{n(n+1)}{2}$ variables  are fixed and note that there are $p^{\frac{n(n+1)}{2}(r-1)}$ ways to fix them. For each such choice Hensel's lemma provides exactly one way to complete the solution $\L'$.  Combining this with estimate \eqref{jewur8t75865iogkjjbhwyt56231830eiooerjgiotvljiogu} produces the upper bound $p^{\frac{n(n+1)}{2}r}(1+\frac{C}{p^2})$ for the number of solutions of \eqref{Snew9}. This finishes the proof of Proposition \ref{propnondiv}.
\bigskip

Next we analyze the case of those primes $p$ which divide  $\det(\Lambda)$. Since there are 
\begin{equation}
\label{nottoomanyefgryft77856t785687}
O(\frac{\log \det(\Lambda)}{\log\log \det(\Lambda)})
\end{equation} 
 such primes, we will content ourselves with obtaining cruder bounds for the densities $\nu_p$, which are only sharp up to a multiplicative constant. We will denote by $o_p(T)$ the largest $\alpha$ such that $p^\alpha\,|\,T$.

One of our main tools here is the following result in \cite{Wo}
\begin{lemma}
\label{lemawooley}
Let $f_1,\ldots,f_d\in\Z[X_1,\ldots,X_d]$ be polynomials of degrees $k_1,\ldots,k_d$ and set
$$J({\bf X})=\det (\frac{\partial f_j}{\partial X_i}({\bf X}))_{1\le i,j\le d}.$$
Then the number of solutions ${\bf X}\in(\Z_{p^r})^d$ of
$$f_j({\bf X})\equiv 0\mod p^r,\;\;1\le j\le d$$
for which $J({\bf X})\not\equiv 0\mod p$ is at most $k_1k_2\ldots k_d$.
\end{lemma}

For an $n\times n$ matrix $\Lambda$ and for $A,B\subset \{1,\ldots,n\}$ with $|A|=|B|$ we define
$$\mu_{A,B}=\det((\Lambda_{i,j})_{i\in A,j\in B}).$$

\begin{proposition}
\label{thecaseofpdivisor}
Let $\Lambda\in M_{n,n}(\Z)$ be a  positive definite matrix and let $p|\det(\Lambda)$. Then
$$\nu_p(I_{n+1},\Lambda)\lesssim \sum_{0\le l_i:1\le i\le n\atop{l_1+l_2+\ldots+l_n\le o_p(\det(\Lambda))}}p^{\beta_2(l_1,\ldots,l_n)+\ldots+\beta_n(l_1,\ldots,l_n)},$$
where $\beta_i=\beta_i(l_1,\ldots,l_n)$ satisfies
$$\beta_i=\min\{(i-1)l_i,(i-2)l_i+\min_{|A|=1}o_p(\mu_{\{1\},A})-l_1,(i-3)l_i+\min_{|A|=2}o_p(\mu_{\{1,2\},A})-l_1-l_2,\ldots,$$$$\ldots,\min_{|A|=i-1}o_p(\mu_{\{1,2,\ldots,i-1\},A})-l_1-l_2-\ldots-l_{i-1}\}$$
\end{proposition}
\begin{proof}
We first show how to count the non-degenerate solutions  for
\begin{equation}
\label{Thfinal1}
\L\in M_{n+1,n}(\Z_{p^r}):\;\L^*\L\equiv \Lambda \mod p^r,
\end{equation}
by which we mean the solutions $(\x^1,\ldots,\x^n)$ such that $\x^j$ are linearly independent in the vector space $\Z_{p^r}^{n+1}$ over the field $\Z_p$. Recall that this implies that the $\frac{n(n+1)}{2}\times n(n+1)$ matrix $M=M(\x^1,\ldots,\x^n)$ has rank$\mod p$ equal to $\frac{n(n+1)}{2}$. Pick $\frac{n(n+1)}{2}$ independent columns of $M$. We fix $\mod p^r$ the values of the $\frac{n(n+1)}{2}$ variables  corresponding to the remaining columns of $M$, and apply Lemma \ref{lemawooley} with $d=\frac{n(n+1)}{2}$ to get at most $O(1)$ solutions. Thus the overall contribution of the non-degenerate solutions is $O( p^{\frac{n(n+1)}2r})$.
\bigskip

We next use a sequence of reductions that will allow us to relate the number of degenerate solutions  to the number of non-degenerate ones. The analysis will be split into $n$ stages.

In the first stage, let us count the solutions for \eqref{Thfinal1} satisfying $o_p(\x^1)=l_1$ for some fixed $l_1\ge 0$. By that we mean that $l_1$ is the largest integer such that $p^{l_1}|x_i^1$ for each $1\le i\le n+1$. We can work with $r$ large enough so that $r\ge 2l_1+1$.  Write $\x^1=p^{l_1}\tilde{\x}^1$ where $\tilde{\x}^1\not\equiv 0\mod p$. Note that the entry $\Lambda_{1,1}$ must be divisible by $p^{2l_1}$ since it equals $\x^1\cdot\x^1\mod p^r$. Similarly, $\Lambda_{1,j}$ must be divisible by $p^{l_1}$ for $j\ge 2$.   Setting
\begin{equation}
\label{Snew10}
\Lambda_{1,1}=p^{2l_1}\tilde{\Lambda}_{1,1},\;\;\;\Lambda_{1,j}=p^{l_1}\tilde{\Lambda}_{1,j},\;j\ge 2
\end{equation}
we derive the new system of congruences with $(\tilde{\x}^1,\x^2,\ldots,\x^n)\in \Z_{p^{r-l_1}}^{n+1}\times\Z_{p^{r}}^{n+1}\times\ldots\times \Z_{p^{r}}^{n+1}$
$$(a)\;\; \tilde{\x}^1\cdot\tilde{\x}^1\equiv \tilde{\Lambda}_{1,1}\mod p^{r-2l_1}$$
$$(b)\;\;\tilde{\x}^1\cdot \x^j\equiv \tilde{\Lambda}_{1,j}\mod p^{r-l_1},\;j\ge 2$$
$$(c)\;\;\x^i\cdot \x^j\equiv \Lambda_{i,j}\mod p^{r}, i,j\ge 2.$$
Note that we also require $o_p(\tilde{\x}^1)=0$.
\bigskip

We argue that the number of solutions to the above system can be bounded by the maximum over all $0\le \Lambda_{1,1}',\Lambda_{1,2}',\ldots,\Lambda_{1,n}'\le p^r-1$ satisfying
\begin{equation}
\label{Snew18}
\Lambda_{1,1}'\equiv \tilde{\Lambda}_{1,1}\mod p^{r-2l_1},\;\Lambda_{1,j}'\equiv \tilde{\Lambda}_{1,j}\mod p^{r-l_1},\;j\ge 2
\end{equation}
of the number of solutions  of the system \begin{equation}
\label{Snew12}\L^*\L=\Lambda'\mod p^r:\;\L=(\x^1,\ldots,\x^n)\in M_{n+1,n}(\Z_{p^r}),\;o_p(\x^1)=0.\end{equation}
Here $\Lambda'$ is the symmetric matrix whose entries $\Lambda_{1,j}'=\Lambda_{j,1}'$ have been defined in \eqref{Snew18}, while we set $\Lambda_{i,j}':=\Lambda_{i,j}$ for the remaining pairs $(i,j)$.
To see this we first note that the system $(a)-(c)$ has $p^{(n+1)l_1}$ fewer solutions than the same system where $\tilde{\x}^1\in \Z_{p^{r-l_1}}^{n+1}$ is replaced with $\x^1\in \Z_{p^{r}}^{n+1}$ (we keep all the modular conditions unchanged). This follows since each $\tilde{x}_i^1\in \Z_{p^{r-l_1}}$ can be lifted in exactly $p^{l_1}$ ways to some $x_i^1\in \Z_{p^{r}}$ with $\tilde{x}_i^1\equiv x_i^1\mod p^{r-l_1}$. Now the number of solutions to this new system is the sum over all $\Lambda_{1,1}',\ldots,\Lambda_{1,n}'$ as in \eqref{Snew18}  of the number of solutions of the system \eqref{Snew12}. It now suffices to note that there are exactly $p^{(n+1)l_1}$ choices for $\Lambda_{1,1}',\ldots,\Lambda_{1,n}'$, and to use the fact that the average is bounded by the maximum.
\bigskip

 In the second stage of our reduction we fix $0\le \Lambda_{1,1}',\dots,\Lambda_{1,n}'\le p^{r}-1$ as in \eqref{Snew18} and count the number of solutions  for \eqref{Snew12}. It suffices to focus attention on those particular solutions for which $o_p(\x^1\wedge\x^2)=l_2$ for fixed $l_2\ge 0$. By that we mean that $l_2$ is the largest integer such that $p^{l_2}$ divides the determinant of all the $2\times 2$ minors of the $(n+1)\times 2$ matrix $[\x^1,\x^2]$. It follows that there must exist $0\le t_{2,1}\le p^{l_2}-1$ and $\tilde{\x}^2\in\Z^{n+1}$ such that
\begin{equation}
\label{Snew11}
\x^2=t_{2,1}\x^1+p^{l_2}\tilde{\x}^2,\text{ with }o_p(\x^1\wedge \tilde{\x}^2)=0.
\end{equation}
Of course, to get this one relies crucially on the fact that $o_p(\x^1)=0$. Also, we allow for $l_2$ to be $0$, in which case we can take $t_{2,1}=0$, $\tilde{\x}^2=\x^2$.

It suffices again to only consider $r\ge 2l_2+1$.
%It is immediate that
%$$\Lambda_{1,2}'\equiv t\Lambda_{1,1}',\;\Lambda_{2,2}'\equiv t\Lambda_{1,2}',\;\Lambda_{2,3}\equiv t\Lambda_{1,3}'\mod p^{l'}.$$
%This in turn leaves at most $\gcd(p^{l'},\Lambda_{1,1}',\Lambda_{1,2}',\Lambda_{1,3}')$ possibilities for the value of $t$.

Fix $0\le t_{2,1}\le p^{l_2}-1$. We reformulate the system \eqref{Snew12} using the variables $(\x^1,\tilde{\x}^2,\ldots,\x^n)\in \Z_{p^r}^{n+1}\times\Z_{p^{r-l_2}}^{n+1}\ldots\times\Z_{p^{r}}^{n+1}$, with $o_p(\x^1)=o_p(\x^1\wedge \tilde{\x}^2)=0$
$$(a')\;\; \tilde{\x}^2\cdot\tilde{\x}^2\equiv {\bar{\Lambda}}_{2,2}\mod p^{r-2l_2}$$
$$(b')\;\; \x^i\cdot\tilde{\x}^2\equiv {\bar{\Lambda}}_{i,2}\mod p^{r-l_2},\;i\not=2$$
$$(c')\;\; \x^i\cdot\x^j\equiv {\Lambda}_{i,j}'\mod p^{r},\;i,j\not=2$$
where
\begin{equation}
\label{Snew16}\begin{cases}
\bar{\Lambda}_{i,2}=p^{-l_2}(\Lambda_{i,2}'-t_{2,1}\Lambda_{i,1}'),\;i\not=2&\\
\bar{\Lambda}_{2,2}=p^{-2l_2}(\Lambda_{2,2}'+t_{2,1}^2\Lambda_{1,1}'-2t_{2,1}\Lambda_{1,2}')
\end{cases}\end{equation}
Reasoning as we did in the previous stage, the number of solutions of the system $(a')-(c')$ is bounded by the maximum over all $0\le \Lambda_{1,2}'',\Lambda_{2,2}'',\ldots, \Lambda_{n,2}''\le p^r-1$ satisfying
\begin{equation}
\label{Snew17}
\Lambda_{2,2}''\equiv \bar{\Lambda}_{2,2}\mod p^{r-2l_2},\;\Lambda_{i,2}''\equiv \bar{\Lambda}_{i,2}\mod p^{r-l_2}\text{ for }i\not=2
\end{equation}
of the number of solutions  of the system
\begin{equation}
\label{Snew13}
\L^*\L=\Lambda''\mod p^r:\;\L=(\x^1,\dots,\x^n)\in M_{n+1,n}(\Z_{p^r}),\;o_p(\x^1\wedge \x^2)=0.
\end{equation}
Here $\Lambda''$ is the symmetric matrix whose entries $\Lambda_{2,i}''=\Lambda_{i,2}''$ have been defined in \eqref{Snew17}, while $\Lambda_{i,j}'':=\Lambda_{i,j}'$ for the remaining pairs $(i,j)$.

Before we go to the next stage, we bound the number of possible values for $t_{2,1}$. First, we have the trivial bound $p^{l_2}$. Also, since $$\Lambda_{1,2}'\equiv \x^2\cdot \x^1\mod p^r,\; \x^1\cdot\x^1\equiv \Lambda_{1,1}'\mod p^r$$ and since $p^{l_2}\tilde{\x}^2$ is determined$\mod p^r$, it follows that $t\Lambda_{1,1}'$ is determined$\mod p^r$. But given the fact that $\Lambda_{1,1}'$ is determined$\mod p^{r-2l_1}$, it follows that $t_{2,1}$ is determined modulo $ p^{r-o_p(\Lambda_{1,1})-2l_1}$. Since also $0\le t_{2,1}\le p^r-1$, we get the upper bound $p^{o_p(\Lambda_{1,1})-2l_1}$ for the number of admissible values of $t_{2,1}$. A very similar reasoning will also produce the bound $p^{o_p(\Lambda_{1,j})-l_1}$, for $j\ge 2$. Combining the two bounds we get an upper bound $p^{\beta_2}$ for the number of admissible values of $t_{2,1}$, where
$$\beta_2=\min\{l_2,\min_{1\le j\le n}o_p(\Lambda_{1,j})-l_1\}$$
\bigskip

We now begin the third stage of the reduction, which we hope will completely clarify the process. We will as before look for  solutions for \eqref{Snew13} which in addition satisfy $o_p(\x^1\wedge\x^2\wedge\x^3)=l_3$ for fixed $l_3\ge 0$. We can write
\begin{equation}
\label{Snew15}
\x^3=t_{3,1}\x^1+t_{3,2}\x^2+p^{l_3}\tilde{\x}^3
\end{equation}
where $o_p(\x^1\wedge\x^2\wedge\tilde{\x}^3)=0$ and
 $0\le t_{3,1},t_{3,2}\le p^{l_3}-1$.

For such a solution to exist it must be that $\det(\Lambda'')\equiv 0\mod p^{l_3}$.
Using \eqref{Snew10}, \eqref{Snew18}, \eqref{Snew16} and \eqref{Snew17} we easily get that 
\begin{equation}
\label{Snew14}
l_1+l_2+l_3\le o_p(\det(\Lambda)).
\end{equation}

Using \eqref{Snew15} we get for each $1\le i\le n$
\begin{equation}
\label{Snew19}
\Lambda_{i,3}''=t_{3,1}\Lambda_{i,1}''+t_{3,2}\Lambda_{i,2}''\mod p^{l_3} 
\end{equation}

We now show how to bound the number of admissible pairs $(t_{3,1},t_{3,2})$. First, there is the trivial bound $p^{2l_3}$.
 
Fix $1\le i\not=j\le n$. We prove that the two equations \eqref{Snew19} for $i$ and $j$ determine the pair $(t_{3,1},t_{3,2})\in \Z_{p^{l_3}}\times \Z_{p^{l_3}}$ up to at most $p^\alpha$ choices, where $p^\alpha$ is the largest power of $p$ that divides $\det\begin{bmatrix}\Lambda_{i,1}''&\Lambda_{i,2}''\\ \Lambda_{j,1}''&\Lambda_{j,2}''
\end{bmatrix}$.
We can assume this determinant to be nonzero, otherwise there is nothing to prove. Thus $(\Lambda_{i,1}'',\Lambda_{i,2}'')\not=(0,0)$, and write $(\Lambda_{i,1}'',\Lambda_{i,2}'')=p^o\u$, where $\u=(u_1,u_2)$ satisfies $o_p(\u)=0$ and $o\le \alpha$. We can assume that $p^o|\Lambda_{j,1}''$ and $p^o| \Lambda_{j,2}''$, otherwise we do the argument for $(\Lambda_{i,1}'',\Lambda_{i,2}'')$ instead. Note that in particular  $o\le \alpha/2$.
We can now write as before $$(\Lambda_{j,1}'',\Lambda_{j,2}'')=w\u+p^{\alpha-o}\v,$$
for some $w$ and $\v=(v_1,v_2)$ with $\begin{bmatrix}u_1&u_2\\ v_1&v_2
\end{bmatrix}$ nonsingular $\mod p$.
Note that $w$ must be divisible by $p^o$.
We thus get that
$$\begin{cases}u_1t_{3,1}+u_2t_{3,2}\equiv 0\mod p^{l_3-o} \\v_1t_{3,1}+v_2t_{3,2}\equiv 0\mod p^{l_3-\alpha+o} \hfill \end{cases}.
$$
Since $\begin{bmatrix}u_1&u_2\\ v_1&v_2
\end{bmatrix}$ is nonsingular $\mod p$, the pair $(t_{3,1},t_{3,2})$ will be uniquely determined in $\Z_{p^{l_3-o}}\times\Z_{p^{l_3-\alpha+o}}$. Note that this can be lifted in exactly $p^\alpha$ ways to a  $\Z_{p^{l_3}}\times \Z_{p^{l_3}}$ pair, which proves the claim. It is easy to see as before that
$$\alpha\le o_p(\det\begin{bmatrix}\Lambda_{i,1}&\Lambda_{i,2}\\ \Lambda_{j,1}&\Lambda_{j,2}
\end{bmatrix})-l_1-l_2.$$
Finally, by fixing  $0\le t_{3,2}\le p^{l_3}-1$,   \eqref{Snew19} for a fixed $i$ will determine the value of $t_{3,1}$ within $p^{o_p(\Lambda_{i,1})-l_1}$ possibilities. 

Combining all three bounds derived above we get the upper bound $p^{\beta_3}$ for the number of pairs $(t_{3,1},t_{3,2})$, where 
$$\beta_3=\min\{2l_3,\min_{1\le j\le n}o_p(\Lambda_{1,j})+l_3-l_1,\min_{|A|=2}o_p(\mu_{\{1,2\},A})-l_2-l_1 \}.$$

It is now clear how to complete the remaining stages of the reduction. In the end we are left with counting non-degenerate solutions corresponding to fixed values of $l_i,t_{i,j}$. As shown in the beginning of the proof, we have the bound $O(p^{\frac{n(n+1)}2r})$ for the number of these solutions.
Also, the computations behind \eqref{Snew14} easily extend to prove 
$$l_1+\ldots+l_n\le o_p(\det \Lambda).$$
The bound for the number of admissible tuples $(t_{i,1},\ldots,t_{i,i-1})$ will follow as indicated before.
This ends the proof of the proposition.
\end{proof}

\subsection{The four dimensional case}

We start by proving  part (a) of Theorem \ref{nSiegel}.  Part (b) will be discussed in the next subsection.

Note that now $m=4$, $n=3$, $\gamma=I_4$,
$$\Lambda=\Lambda_{a,b}=\begin{bmatrix}\lambda&a&\lambda+a-b\\a&\lambda&b\\ \lambda+a-b&b&\lambda\end{bmatrix}$$
 and
$$\nu_p=\nu_p(I_4,\Lambda)=\lim_{r\to\infty}\frac1{p^{6r}}|\{\L\in M_{4,3}(\Z_{p^r}):\;\L^*\L\equiv \Lambda \mod p^r\}|.$$
We will spend the rest of this subsection mainly evaluating $\nu_p$. We will be interested only in values of $a,b$ for which the equation $\L^*\L=\Lambda_{a,b}$ has at least one solution $\L$. In this case, it will be immediate that $\Lambda_{a,b}$ is positive semi-definite, and in fact positive definite if its determinant $2(b-\lambda)(a+\lambda)(a-b)$ is not zero. But then \eqref{Snew1863}  will imply that
\begin{equation}
\label{Snew22}
N_{a,b,\lambda}\lesssim \prod_{p\text{ prime}}\lim_{r\to\infty}\frac1{p^{6r}}|\{\L\in M_{4,3}(\Z_{p^r}):\;\L^*\L\equiv \Lambda_{a,b} \mod p^r\}|.
\end{equation}

We first note the easy estimate which takes care of the singular case

\begin{equation}
\label{Snew23}
\sum_{|a|,|b|\le \lambda:\atop{a=b\text{ or }b=\lambda\text{ or }a=-\lambda}}N_{a,b,\lambda}\lesssim \lambda^{2+\epsilon}.
\end{equation}
Let us see the $a=b$ case, the other two cases are very similar. Note that if $\L=(\x,\y,\z)$ satisfies $\L^*\L=\Lambda_{a,a}$ for some $a$ then
$$\x\cdot (\x-\z)=\y\cdot (\x-\z)=\z\cdot (\x-\z)=0,$$
which immediately implies that $\x,\y,\z$ are linearly dependent. If $\x$ and $\y$ are fixed, then $\z$ is hence constrained to a circle on $\F_{4,\lambda}$ and can only take $O(\lambda^{\epsilon})$ values. Note also that since $\x,\y\in \F_{4,\lambda}$, there are $O(\lambda^{2+\epsilon})$ such pairs $(\x,\y)$. 
\bigskip

We next focus on the nonsingular case.
An immediate consequence of Proposition \ref{thecaseofpdivisor} is
\begin{proposition}
\label{Pr:divisorcase}
Assume $\lambda\not\in\{-a,b\}$ and $a\not=b$. If $p\,|\,\det(\Lambda_{a,b})$ then
$$\nu_p\lesssim o_p(\det(\Lambda_{a,b}))^2p^{o_p(\gcd(\lambda^2-a^2,\lambda^2-b^2))}.$$
\end{proposition}
\begin{proof}
First use the bounds $\beta_2\le l_2$, $\beta_3\le o_p(\lambda^2-a^2)-l_1-l_2$ and sum over $l_1+l_2+l_3\le o_p(\det(\Lambda_{a,b}))$ to get
$$\nu_p\lesssim o_p(\det(\Lambda_{a,b}))^2p^{o_p(\lambda^2-a^2)}.$$
Then repeat the argument with indices $1,2$ replaced by $2,3$.
\end{proof}

We can now prove
\begin{corollary}
\label{finalcor6689}
We have
$$\sum_{|a|,|b|\le  \lambda:\atop{a\not=b,\lambda\not\in\{-a,b\}}}N_{a,b,\lambda}\lesssim \lambda^{2+\epsilon}.$$
\end{corollary}
\begin{proof}Using \eqref{Snew22}, \eqref{nottoomanyefgryft77856t785687}, Proposition \ref{propnondiv} and Proposition \ref{Pr:divisorcase} we get that

\begin{equation}
\label{Snew28}
N_{a,b,\lambda}\lesssim \lambda^{\epsilon}\gcd(\lambda^2-a^2,\lambda^2-b^2).
\end{equation}
Fix $1\le d\le\lambda$. We rely on the fact that the number of divisors of an integer $l$ is $O(l^{\epsilon})$. Then
$$|\{|a|<\lambda:d\,|\,\lambda^2-a^2\}|\le\sum_{s:\;s^2|d}\sum_{d_1,d_2\ge 1\atop{d_1d_2=d,\;\gcd(d_1,d_2)=s}}|\{|a|<\lambda:d_1\,|\,\lambda-a,\;d_2\,|\,\lambda+a\}|\le$$
$$\sum_{s:\;s^2|d\atop{s|2\lambda}}\sum_{d_1,d_2\ge 1\atop{d_1d_2=d,\;\gcd(d_1,d_2)=s}}|\{(k_1,k_2)\in\Z^2:\;1\le k_i\le \frac{2\lambda}{d_i},\;k_1d_1+k_2d_2=2\lambda\}|=$$
$$\sum_{s:\;s^2|d\atop{s|2\lambda}}\sum_{\bar{d}_1,\bar{d}_2\ge 1\atop{s^2\bar{d}_1\bar{d}_2=d,\;\gcd(\bar{d}_1,\bar{d}_2)=1}}|\{(k_1,k_2)\in\Z^2:\;1\le k_i\le \frac{2\lambda}{d_i},\;k_1\bar{d}_1+k_2\bar{d}_2=2\lambda/s\}|\le$$
$$\sum_{s:\;s^2|d\atop{s|2\lambda}}\sum_{\bar{d}_1,\bar{d}_2\ge 1\atop{s^2\bar{d}_1\bar{d}_2=d,\;\gcd(\bar{d}_1,\bar{d}_2)=1}}\frac{2\lambda}{d_1\bar{d}_2}\le \sum_{s:\;s^2|d\atop{s|2\lambda}}\sum_{\bar{d}_1|d}\frac{2s\lambda}{d}\lesssim \lambda^{\epsilon}\sum_{s:\;s^2|d\atop{s|2\lambda}}\frac{2s\lambda}{d}.$$
We now conclude that
$$\sum_{|a|,|b|\le \lambda}\gcd(\lambda^2-a^2,\lambda^2-b^2)\le \sum_{1\le d\le \lambda}d(2+|\{|a|<\lambda:d\,|\,\lambda^2-a^2\}|)^2\lesssim$$
$$ \sum_{1\le d\le \lambda}d+\lambda^{\epsilon}\sum_{1\le d\le \lambda}d\sum_{s:\;s^2|d\atop{s|2\lambda}}\frac{s^2\lambda^2}{d^2}\lesssim \lambda^2+\lambda^{\epsilon}\sum_{s|\lambda}\sum_{1\le d\le \lambda\atop{s^2|d}}\frac{s^2\lambda^2}{d}\lesssim \lambda^{2+
\epsilon}.$$
\end{proof}
Note that Corollary \ref{finalcor6689} combined with \eqref{Snew23}  proves part (a) of Theorem \ref{nSiegel}.
\bigskip

\subsection{The five dimensional case}

We now discuss the proof  of Theorem \ref{nSiegel} (b). In this case $m=5$, $n=4$, $\gamma=I_5$,
$$\Lambda=\Lambda_{a,b,c,d}=\begin{bmatrix}a&c&
a/2&a/2\\c&b&
b/2&b/2\\a/2&b/2&
\lambda&d\\a/2&b/2&
d&\lambda\end{bmatrix}.$$
and
$$\nu_p=\nu_p(I_5,\Lambda)=\lim_{r\to\infty}\frac1{p^{10r}}|\{\L\in M_{5,4}(\Z_{p^r}):\;\L^*\L\equiv \Lambda \mod p^r\}|.$$

We first analyze solutions corresponding to the degenerate cases, by which we mean $a\in\{0,4\lambda\}$ or $b\in\{0,4\lambda\}$ or $a=b$ or $\det(\Lambda_{a,b,c,d})=0$.
Note first that each solution $(\u,\v,\x,\y)$ counting towards some $N_{a,b,c,d,\lambda}$ with $\det(\Lambda_{a,b,c,d})=0$ will necessarily satisfy $2\le \text{rank}[\u,\v,\x,\y]<4$.
The computations in section \ref{sec5} combined with \eqref{sharpestnumberlatt234}, \eqref{sharpestnumberlatt} show that the sum over $a,b,c,d$ of the number of all solutions with $\text{rank}[\u,\v,\x,\y]=2$ is in fact bounded by
$$\sum_{\x,\y\in \F_{5,\lambda}\atop{\x\not=\y}}|\{(\u,\v):\u,\v\in (\x+\F_{5,\lambda})\cap(\y+\F_{5,\lambda})\cap \langle \x, \y\rangle\}|$$$$\le\sum_{\x,\y\in \F_{5,\lambda}\atop{\x\not=\y}}|\{\u:\u\in (\x+\F_{5,\lambda})\cap \langle \x, \y\rangle\}|^2\lesssim \lambda^\epsilon|\F_{5,\lambda}|^2\lesssim \lambda^{3+\epsilon}.$$

To count the solutions with $\text{rank}[\u,\v,\x,\y]=3$, note first that since $\x,\y$ are linearly independent, we must have $\text{rank}[\u,\x,\y]=3$ or $\text{rank}[\v,\x,\y]=3$. By symmetry we focus on the first case. Reasoning as before, the sum over $a,b,c,d$ of the number of all such solutions is bounded by
$$\sum_{\x,\y\in \F_{5,\lambda}\atop{\x\not=\y}}\sum_{\u\in (\x+\F_{5,\lambda})\cap(\y+\F_{5,\lambda})} |\langle \x, \y,\u\rangle\cap (\x+\F_{5,\lambda})\cap(\y+\F_{5,\lambda})|\lesssim$$
$$|\F_{5,\lambda}|^2\lambda^{1+\epsilon}\lesssim \lambda^{4+\epsilon}.$$

Next, we count the solutions corresponding to the remaining degenerate cases, under the additional assumption that now $\det(\Lambda_{a,b,c,d})\not=0$. If $a=0$ then $\u=\textbf{0}$, so this corresponds to a zero determinant. If $a=4\lambda$ then $|\u|=2\sqrt{\lambda}$. Note that in addition $\u\in \x+\F_{5,\lambda}\subset \F_{5,\lambda}+\F_{5,\lambda}$ and these force $\x=\u$. Such a solution is again excluded, since it corresponds to a singular $\Lambda_{a,b,c,d}$.

To close the analysis of the degenerate cases,  we count the contribution from the $a=b$ case. Note that we must have $|\u|=|\v|$ and $\x\cdot(\u-\v)=\y\cdot(\u-\v)=0$. The corresponding contribution is bounded by
$$\sum_{\x,\y\in \F_{5,\lambda}\atop{\x\not=\y}}\sum_{\u\in (\x+\F_{5,\lambda})\cap(\y+\F_{5,\lambda})}|\{\v\in\Z^5:|\u|=|\v|,\v\in\u+\langle\x,\y\rangle^{\perp}\}|\lesssim$$
$$\lambda^{1/2+\epsilon}\sum_{\x,\y\in \F_{5,\lambda}}|(\x+\F_{5,\lambda})\cap(\y+\F_{5,\lambda})|=\lambda^{1/2+\epsilon}\E(\F_{5,\lambda})\lesssim \lambda^{4+\epsilon},$$
where for the last inequality we used \eqref{energy5dopttt}.
\bigskip

We begin the analysis of the non-degenerate case by recording the following consequence of Proposition \ref{propnondiv} 
\begin{proposition}
\label{propnondivfivedddd}
Assume $p$ is not a factor of $\det (\Lambda_{a,b,c,d})\not=0$. Then
$$\nu_p\le 1+\frac{C}{p^2},$$
where $C$ is independent of $p,a,b,c,d,\lambda.$
\end{proposition}
Note as before that the product of these $\nu_p$ is $O(1)$.
\bigskip

Also, by using the bounds $$\beta_2\le l_2$$
$$\beta_3\le l_3+\min_{|A|=1}o_p(\mu_{\{1\},A})-l_1$$
$$\beta_4\le \min_{|A|=3}o_p(\mu_{\{1,2,3\},A})-l_1-l_2-l_3$$ in 
Proposition \ref{thecaseofpdivisor}, and by permuting indices we get
\begin{proposition}
\label{Pr:divisorcasefived}
Assume $\Lambda:=\Lambda_{a,b,c,d}$ is nonsingular and $p|\det(\Lambda)$. Then
$$\nu_p\lesssim o_p(\det(\Lambda))^3p^{\min_{1\le i,j\le 4\atop{A,B\subset\{1,2,3,4\}\atop{i\in A,|A|=|B|=3}}}(o_p(\Lambda_{i,j})+o_p(\mu_{A,B}))}.$$
\end{proposition}
Since two of the rows of $\Lambda$ contain (divisors of)  both $a$ and $\lambda$, we conclude
$$\nu_p\lesssim o_p(\det(\Lambda))^3p^{o_p(\gcd(a,\lambda))+\min_{A,B\subset\{1,2,3,4\}\atop{|A|=|B|=3}}o_p(\mu_{A,B})}.$$

By using various choices for $A,B$, then invoking Propositions \ref{propnondivfivedddd}, \ref{Pr:divisorcasefived} and equations \eqref{nottoomanyefgryft77856t785687}, \eqref{Snew1863}
we conclude that whenever $\det (\Lambda_{a,b,c,d})\not=0$
$$N_{a,b,c,d,\lambda}\lesssim\lambda^{\epsilon}\gcd(a,\lambda)\cdot$$$$\cdot\gcd(\lambda(ab-c^2)+\frac{ab}{4}(2c-a-b),d(ab-c^2)+\frac{ab}{4}(2c-a-b),a(\lambda-d)(b-c),b(\lambda-d)(a-c))\le$$
$$\le\lambda^{\epsilon}\gcd(a,\lambda)\gcd(\eqref{Snew29},\eqref{Snew30},\eqref{Snew31},\eqref{Snew32})$$
where
\begin{equation}
\label{Snew29}
\lambda(ab-c^2)+\frac{ab}{4}(2c-a-b)
\end{equation}\begin{equation}
\label{Snew30}
(\lambda-d)(ab-c^2)
\end{equation}\begin{equation}
\label{Snew31}
a(\lambda-d)(b-c)\end{equation}\begin{equation}
\label{Snew32}
b(\lambda-d)(a-c).\end{equation}
\bigskip

We use $\A$ to denote the non degenerate four-tuples $(a,b,c,d)$ with  $|a|,|b|,|c|,|d|\lesssim \lambda$, $a,b\notin\{0,4\lambda\}$, $a\not=b$ and $\det(\Lambda_{a,b,c,d})\not=0$.
To finish the proof of part (b) of Theorem \ref{nSiegel} we are left with evaluating
$$\sum_{(a,b,c,d)\in\A}\gcd(a,\lambda)\prod_{p|\lambda}p^{o_p(\gcd(\eqref{Snew29},\eqref{Snew30},\eqref{Snew31},\eqref{Snew32}))}\prod_{p\nmid\lambda}p^{o_p(\gcd(\eqref{Snew29},\eqref{Snew30},
\eqref{Snew31},\eqref{Snew32}))}.$$
This  can be trivially bounded by
$$\sum_{k_0|\lambda\atop{k_1\lesssim \lambda^3,\tilde{k}_1|\lambda\atop{k_2\lesssim \lambda^3,(k_2,\lambda)=1}}}k_0k_1k_2\sum_{(a,b,c,d)\in\A:\atop{k_0|a\atop{k_1|\gcd(\eqref{Snew29},\eqref{Snew30},\eqref{Snew31},\eqref{Snew32})\atop{k_2|\gcd(\eqref{Snew29},
\eqref{Snew30},\eqref{Snew31},\eqref{Snew32})}}}}1.$$
Here and in the future we denote by $\tilde{k}$ the product of all primes dividing $k$. Note that if $(a,b,c,d)\in\A$ then the two terms \eqref{Snew31}-\eqref{Snew32} can not be zero simultaneously, which justifies the finiteness restriction $k_1,k_2\lesssim \lambda^3$. Indeed, $d\not=\lambda$ since $\det(\Lambda_{a,b,c,d})\not=0$. Also $a,b\not=0$ and moreover $b-c$ and $a-c$ can not be both zero since we have assumed $a\not=b$.

Note that the number of integers $k_0$ and $k_1$ in the sum is $O(\lambda^{\epsilon})$ and that $(k_0k_1,k_2)=1.$
Since $k_1|(\lambda-d)b(a-c)$, there exists a decomposition
\begin{equation}
\label{Snew34}
k_1=k_1'k_1''k_1'''
\end{equation}
 with
\begin{equation}
\label{Snew40}
k_1'|\lambda-d,\;\;k_1''|b,\;\;k_1'''|a-c.
\end{equation}
We further bound the sum by
$$\sum_{k_0|\lambda\atop{k_1\lesssim \lambda^3,\tilde{k}_1|\lambda}}k_0k_1\sum_{k_1',k_1'',k_1'''\atop{k_1=k_1'k_1''k_1'''}}\sum_{a,b:\atop{k_0|a,k_1''|b}}\sum_{k_2\lesssim \lambda^3:\atop{(k_2,\lambda)=1}}k_2\sum_{c,d:(a,b,c,d)\in \A\atop{k_1'''|c-a,\;k_1'|d-\lambda\atop{k_2|\gcd(\eqref{Snew29},\eqref{Snew30},\eqref{Snew31},\eqref{Snew32})}}}1.$$
Note that for each $k_2,a,b,c,d$ contributing to the summation there must exist a decomposition $k_2=k_2'k_2''$ such that
\begin{equation}
\label{Snew44}
k_2'|\lambda-d
\end{equation}
\begin{equation}
\label{Snew41}
k_2''|ab-c^2,\;\;k_2''|a(b-c),\;\;k_2''|b(a-c).
\end{equation}
This gives rise to the new bound
\begin{equation}
\label{Snew46}
\sum_{k_0|\lambda\atop{k_1\lesssim \lambda^3,\tilde{k}_1|\lambda}}k_0k_1\sum_{k_1',k_1'',k_1'''\atop{k_1=k_1'k_1''k_1'''}}\sum_{a,b:\atop{k_0|a,k_1''|b}}\sum_{k_2\lesssim \lambda^3:\atop{(k_2,\lambda)=1}}k_2\sum_{k_2',k_2'':\atop{k_2=k_2'k_2''}}\sum_{c,d:(a,b,c,d)\in \A\atop{k_1'''|c-a,\;k_1'|d-\lambda\atop{k_2'|d-\lambda\atop{k_2''|\gcd(ab-c^2,a(b-c),b(a-c))}}}}1.
\end{equation}
Equation \eqref{Snew29} written as a quadratic polynomial in $c$ is
\begin{equation}
\label{Snew33}
4\lambda c^2-2abc+(a^2b+ab^2-4\lambda ab)\equiv 0 \mod k_2.
\end{equation}
The discriminant is $D_{a,b}=4ab(4\lambda-a)(4\lambda-b)$. For $a,b$ fixed we run a second decomposition for $k_2$ namely \begin{equation}
\label{Snew36}k_2=k_{2,1,a,b}k_{2,2,a,b}
\end{equation}
 with  \begin{equation}
\label{Snew37}
(k_{2,1,a,b},k_{2,2,a,b})=1,\;(k_{2,1,a,b},D_{a,b})=1,\;\tilde{k}_{2,2,a,b}|D_{a,b}.
\end{equation}
Note that this decomposition is unique, since $D_{a,b}\not=0$.

For each $a,b,k_2',k_2''$ as above, define
$$k_{2,1,a,b}'=\gcd(k_2',k_{2,1,a,b}),\;\;\;k_{2,2,a,b}'=\gcd(k_2',k_{2,2,a,b}), $$$$k_{2,1,a,b}''=\gcd(k_2'',k_{2,1,a,b}),\;\;\;k_{2,2,a,b}''=\gcd(k_2'',k_{2,2,a,b}),$$
and note that due to \eqref{Snew34}, \eqref{Snew36}, \eqref{Snew37} we have
$$
k_2'=k_{2,1,a,b}'k_{2,2,a,b}',\;\;\;k_2''=k_{2,1,a,b}''k_{2,2,a,b}''.
$$

Fix now $a,b,k_2',k_2''$. Note that this means that $k_2$ and $k_{2,1,a,b}, k_{2,1,a,b}', k_{2,1,a,b}'',k_{2,2,a,b}, k_{2,2,a,b}'$ and  $k_{2,2,a,b}''$ are also determined. Fix also $k_1'''$. We claim that given these, $c$ will be determined $\mod\frac{k_{2,1,a,b}k_{2,2,a,b}''k_1'''}{\gcd(a,b,k_{2,2,a,b}'')}$ within two possible values. To see this, recall first that \eqref{Snew40} determines $c\mod k_1'''$. Second, the last two divisibilities in \eqref{Snew41} (with $k_{2,2,a,b}''$ replacing $k_2''$) will determine $c\mod \frac{k_{2,2,a,b}''}{\gcd(a,b,k_{2,2,a,b}'')}$. Third, \eqref{Snew33} and \eqref{Snew37} combined with the Chinese remainder Theorem and Hensel's Lemma \ref{lemma:Hensel} determine $c\mod k_{2,1,a,b}$ within 2 possible values. Finally, note that any two of $k_{2,1,a,b},k_{2,2,a,b}'',k_1'''$ are relatively primes, so the claim will follow from the Chinese remainder Theorem. We conclude that given $a,b,k_2',k_2'',k_1'''$, there will be $O(\frac{\lambda\gcd(a,b,k_{2,2,a,b}'')}{k_{2,1,a,b}k_{2,2,a,b}''k_1'''})$ admissible values for $c$ in our summation.

Similarly, \eqref{Snew40} and  \eqref{Snew44} show that for each $k_1',k_2'$ fixed, there are $O(\frac{\lambda}{k_1'k_2'})$ admissible values for $d$. We thus can update the bound on the sum \eqref{Snew46} to
$$\lambda^2\sum_{k_0|\lambda\atop{k_1\lesssim \lambda^3,\tilde{k}_1|\lambda}}k_0k_1\sum_{k_1',k_1'',k_1'''\atop{k_1=k_1'k_1''k_1'''}}\sum_{|a|,|b|\lesssim \lambda:\atop{k_0|a,k_1''|b}}\sum_{k_2\lesssim \lambda^3:\atop{(k_2,\lambda)=1}}k_2\sum_{k_2',k_2'':\atop{k_2=k_2'k_2''\atop{k_2''|ab(a-b)}}}\frac{\gcd(a,b,k_{2,2,a,b}'')}{k_1'k_2'k_{2,1,a,b}k_{2,2,a,b}''k_1'''}=$$
$$\lambda^2\sum_{k_0|\lambda\atop{k_1\lesssim \lambda^3,\tilde{k}_1|\lambda}}\sum_{k_1',k_1'',k_1'''\atop{k_1=k_1'k_1''k_1'''}}\sum_{|a|,|b|\lesssim \lambda:\atop{k_0|a,k_1''|b}}\sum_{k_2\lesssim \lambda^3:\atop{(k_2,\lambda)=1}}\sum_{k_2',k_2'':\atop{k_2=k_2'k_2''\atop{k_2''|ab(a-b)}}}k_0k_1''\frac{\gcd(k_{2,1,a,b},k_2'')}{k_{2,1,a,b}}\gcd(a,b,k_{2,2,a,b}''),$$
where we have used that $\gcd(ab-c^2,a(b-c),b(a-c))|ab(a-b)$ and the various decompositions for $k_1,k_2$.

Since $\gcd(\lambda,k_2)=1$, we can choose a large enough integer $M$ with $\gcd(M,\lambda)=1$ such that $\gcd(a,b,k_{2,2,a,b}'')\le \gcd(a,b,M)$ for all admissible $a,b,k_{2,2,a,b}''$. We can now bound the sum above by
\begin{equation}
\label{Snew61}
\lambda^2\sum_{k_0|\lambda\atop{k_1\lesssim \lambda^3,\tilde{k}_1|\lambda}}k_0\sum_{k_1',k_1'',k_1'''\atop{k_1=k_1'k_1''k_1'''}}k_1''\sum_{|a|,|b|\lesssim \lambda:\atop{k_0|a,k_1''|b}}\gcd(a,b,M)\sum_{k_2\lesssim \lambda^3:\atop{(k_2,\lambda)=1}}\sum_{k_2',k_2'':\atop{k_2=k_2'k_2''\atop{k_2''|ab(a-b)}}}\frac{\gcd(k_{2,1,a,b},k_2'')}{k_{2,1,a,b}}.
\end{equation}
We next fix $a,b$ and evaluate
$$\sum_{k_2\lesssim \lambda^3:\atop{(k_2,\lambda)=1}}\sum_{k_2',k_2'':\atop{k_2=k_2'k_2''\atop{k_2''|ab(a-b)}}}\frac{\gcd(k_{2,1,a,b},k_2'')}{k_{2,1,a,b}}.$$
By using the divisor function bound and the divisibility relations $k_2''|ab(a-b)$ and $\tilde{k}_{2,2,a,b}|D_{a,b}$ we get that $k_2''$ and ${k}_{2,2,a,b}$ are both determined within $O(\lambda^\epsilon)$ values. Call $\A''$ and $\A_{2,2}$ the corresponding set of admissible values. On the other hand, fixing ${k}_{2,2,a,b}\in \A_{2,2}$ and $1\le {k}_{2,1,a,b}\lesssim \lambda^3$ will certainly uniquely determine both $k_2$ and $k_2'$. We thus can write
$$\sum_{k_2\lesssim \lambda^3:\atop{(k_2,\lambda)=1}}\sum_{k_2',k_2'':\atop{k_2=k_2'k_2''\atop{k_2''|ab(a-b)}}}\frac{\gcd(k_{2,1,a,b},k_2'')}{k_{2,1,a,b}}\le$$
$$\sum_{k_2''\in \A''}\sum_{k_{2,2}\in \A_{2,2}}\sum_{1\le k_{2,1}\lesssim \lambda^3}\frac{\gcd(k_{2,1},k_2'')}{k_{2,1}}.$$
Note however that with $k_2''$ fixed,
$$\sum_{1\le k_{2,1}\lesssim \lambda^3}\frac{\gcd(k_{2,1},k_2'')}{k_{2,1}}\le \sum_{d|k_2''}d\sum_{1\le m\lesssim \lambda^3/d}\frac1{md}\lesssim \lambda^\epsilon.$$
Thus
$$\eqref{Snew61}\lesssim \lambda^{2+\epsilon}\sum_{k_0|\lambda\atop{k_1\lesssim \lambda^3,\tilde{k}_1|\lambda}}k_0\sum_{k_1',k_1'',k_1'''\atop{k_1=k_1'k_1''k_1'''}}k_1''\sum_{|a|,|b|\lesssim \lambda:\atop{k_0|a,k_1''|b}}\gcd(a,b,M).$$
Observe now that $k_0,k_1$ and thus also $k_1',k_1'',k_1'''$ can take $O(\lambda^\epsilon)$ values. Note also that $\gcd(a,b,M)\le \gcd(\frac{a}{k_0},\frac{b}{k_1''})$ since $\gcd(M,\lambda)=1$. For fixed $k_0,k_1''$ we estimate
$$k_0k_1''\sum_{|a|,|b|\lesssim \lambda:\atop{k_0|a,k_1''|b}}\gcd(a,b,M)\le k_0k_1''\sum_{|a|,|b|\lesssim \lambda:\atop{k_0|a,k_1''|b}}\gcd(\frac{a}{k_0},\frac{b}{k_1''})\le$$
$$k_0k_1''\sum_{1\le d\lesssim \lambda}\sum_{1\le m_1\lesssim \frac{\lambda}{k_0d}}\sum_{1\le m_2\lesssim \frac{\lambda}{k_1''d}}d\lesssim \lambda^{2+\epsilon}.$$
We can now finish the argument by observing that
$$\lambda^{2+\epsilon}\sum_{k_0|\lambda\atop{k_1\lesssim \lambda^3,\tilde{k}_1|\lambda}}k_0\sum_{k_1',k_1'',k_1'''\atop{k_1=k_1'k_1''k_1'''}}k_1''\sum_{|a|,|b|\lesssim \lambda:\atop{k_0|a,k_1''|b}}\gcd(a,b,M)\lesssim \lambda^{4+\epsilon}.$$

\section{Energy estimates using Siegel's mass formula}
\label{sec5}
In this section we show how Theorem \ref{nSiegel} produces a different type of upper bound for the additive energy. When $n=4$ this method seems to only work for  the whole $\F_{4,\lambda}$.
\begin{theorem}
We have
\begin{equation}
\label{e3e1}
\E(\F_{4,\lambda})\lesssim_\epsilon N^{4+\epsilon}
\end{equation}
\end{theorem}
\begin{proof}
Note that we need to count the number of  quadruples $(\x,\y,\z,\w)\in(\Z^4)^4$ that satisfy
$$\begin{cases}& \x+\y=\z+\w\\ \hfill  &|\x|^2=|\y|^2=|\z|^2=|\w|^2=\lambda.\end{cases}.$$
Since
$$\lambda=|\x+\y-\z|^2=3\lambda+2\x\cdot\y-2\x\cdot \z-2\y\cdot\z,$$
it turns out that
$$\E(\F_{4,\lambda})=|\{(\x,\y,\z)\in(\Z^4)^3:\;|\x|^2=|\y|^2=|\z|^2=\lambda,\;-\x\cdot\y+\x\cdot \z+\y\cdot\z=\lambda\}|=$$
$$ \sum_{a,b\in\Z\atop{|a|,|b|\le\lambda}}|\{(\x,\y,\z)\in(\Z^4)^3:\;|\x|^2=|\y|^2=|\z|^2=
\lambda,\;\x\cdot\y=a,\;\y\cdot \z=b,\; \x\cdot\z=\lambda+a-b\}|.$$
Note that a triple $(\x,\y,\z)$ as above satisfies
$$\begin{bmatrix}x_1&x_2&x_3&x_4\\y_1&y_2&y_3&y_4\\z_1&z_2&z_3&z_4\end{bmatrix}\begin{bmatrix}x_1&y_1&
z_1\\x_2&y_2&
z_2\\x_3&y_3&
z_3\\x_4&y_4&
z_4\end{bmatrix}=\begin{bmatrix}\lambda&a&\lambda+a-b\\a&\lambda&b\\ \lambda+a-b&b&\lambda\end{bmatrix}.$$
Now Theorem \ref{nSiegel} will provide the bound $O(\lambda^{2+\epsilon})$ for the sum above.
\end{proof}
\begin{remark}
\label{jdcvyr7uyr7f67r4}It is easy to see that (apart from $\epsilon$) the bound \eqref{e3e1} is sharp, if no further restriction is placed on $\lambda$. Indeed,
$$\E(\F_{4,\lambda})=\|\sum_{\xi\in\F_{4,\lambda}}e(x\cdot\xi)\|_4^4.$$
 Since $|e(x\cdot\xi)-1|<1/2$ if $|x|\ll \frac1N$ and $\xi\in \F_{4,\lambda}$, it follows that
$$|\sum_{\xi\in\F_{4,\lambda}}e(x\cdot\xi)|\gtrsim |\F_{4,\lambda}|$$if $|x|\ll \frac1N$. It now suffices to choose $\lambda$ such that $|\F_{4,\lambda}|\gtrsim N^2$.
\end{remark}
\vspace{0.2in}

Let us now switch attention to five dimensions. Note that the proof of \eqref{e3e1} combined with the case $m=5,n=3$ in  Siegel's mass formula  proves
\begin{equation}
\label{energy5dopttt}
\E(\F_{5,\lambda})\lesssim N^7.
\end{equation}
Alternatively, one could count the number of solutions of $$\{(\x,\y,\z,\w)\in \F_{5,\lambda}^4:\x+\y=\z+\w\}$$ by fixing $x_5,y_5,z_5$ (there $O(N^3)$ ways) and then applying  the bound in \eqref{e3e1}.

While as observed below \eqref{energy5dopttt} is sharp,  we can gain slightly more by applying Siegel's mass formula with $m=5,n=4$.
\begin{theorem}
\label{efjrg7459t56yi6khjkioh}
For each $\Lambda\subset \F_{5,\lambda}$ we have
\begin{equation}
\label{e3e2}
\E(\Lambda)\lesssim_\epsilon N^{4+\epsilon}|\Lambda|
\end{equation}
\end{theorem}
\begin{proof}
Note that
$$\E(\Lambda)=\sum_{\x,\y\in\Lambda}|(\Lambda+\x)\cap (\Lambda+\y)|\le$$
$$|\Lambda|^2+\Lambda[\sum_{\x\not=\y\in\Lambda}|(\Lambda+\x)\cap (\Lambda+\y)|^2]^{1/2}=$$
$$|\Lambda|^2+|\Lambda||\{(\u,\v,\x,\y)\in \Z^5\times \Z^5\times\F_{5,\lambda}\times \F_{5,\lambda}:\x\not=\y,\u-\x,\v-\x,\u-\y,\v-\y\in\F_{5,\lambda}  \}|^{1/2}.$$
Note that for each $(\u,\v,\x,\y)$ as above we have
$$|\u|^2=2\x\cdot\u=2\y\cdot\u\;\;\;\;\;\;\;\text{and    }\;\;\;\;\;\;|\v|^2=2\x\cdot\v=2\y\cdot\v.$$
It thus follows that
$$|\{(\u,\v,\x,\y)\in \Z^5\times \Z^5\times\F_{5,\lambda}\times \F_{5,\lambda}:\x\not=\y,\u-\x,\v-\x,\u-\y,\v-\y\in\F_{5,\lambda}  \}|$$
$$\le\sum_{a,b,c,d\in \Z\atop{|a|,|b|,|c|,|d|\lesssim \lambda}}|\{(\u,\v,\x,\y)\in (\Z^5)^4:\x\not=\y,|\u|^2=a,|\v|^2=b,|\x|^2=|\y|^2=\lambda, $$$$\u\cdot\v=c,\x\cdot\y=d, \u\cdot\x=\u\cdot\y=\frac{a}2,  \v\cdot\x=\v\cdot\y=\frac{b}2\}|.$$
The result follows  by invoking part (b) of Theorem \ref{nSiegel}.
\end{proof}
\begin{remark}
\label{jdcvyr7uyr7f67r4kk}A computation similar to the one in Remark \ref{jdcvyr7uyr7f67r4} combined with \eqref{sharpestnumberlatt} shows that
$$\E(\F_{5,\lambda})\gtrsim N^7,$$
this time for each $\lambda>0$.
This shows that Theorem \ref{efjrg7459t56yi6khjkioh} is sharp, in the sense that one can not lower one of the exponents in either $N^4$ or $|\Lambda|^1$ from \eqref{e3e2}, without increasing the other one.
\end{remark}
In the next section we will combine the estimates for the energy obtained here with the different type of estimates we have derived using  incidence theory.

\section{Proof of Theorem \ref{mainthm}}
\label{sec6}
Fix $\|a_\xi\|_2=1$ and define $$F(x)=\sum_{\xi\in\F_{n,\lambda}}a_\xi e(x\cdot\xi).$$
We start by recalling the following estimate (24) from \cite{BD}.
\begin{proposition}
\label{earliersupercrit}
For $\alpha>N^{\frac{n-1}{4}+\epsilon}$ we have
\begin{equation}
\label{bnew13fnguithtuirjioju}
|\{|F|>\alpha\}|\lesssim \alpha^{-2\frac{n-1}{n-3}}N^{\frac2{n-3}}.
\end{equation}
\end{proposition}
\vspace{0.2in}

We now work the details for Theorem \ref{mainthm} in the case  $n=4$ and then briefly explain how to modify the argument when $n=5$.
First note that \eqref{partialenergestfourdim} and \eqref{e3e1} imply that
\begin{equation}
\label{e3e3}
\E(\Lambda)\lesssim_\epsilon N^{\frac47+\epsilon}|\Lambda|^2
\end{equation}
for each $\Lambda\subset \F_{4,\lambda}$. This is equivalent with the fact that the linear operator
$$T((a_\xi)_{\xi\in \F_{4,\lambda}})=\sum_{\xi\in\F_{4,\lambda}}a_\xi e(x\cdot \xi)$$
has a  restricted weak type bound  $O(N^{\frac17+\epsilon})$ when acting
$$T:l^2(\F_{4,\lambda})\to L^4(\T^4).$$
In other words, for each $\Lambda\subset \F_{4,\lambda}$ and each $|a_\xi|\le 1_{\Lambda}(\xi)$ we have
$$\|T((a_\xi)_{\xi\in \F_{4,\lambda}})\|_{L^4}\lesssim_\epsilon N^{\frac17+\epsilon}|\Lambda|^{1/2}.$$
It is very easy to convert this into a strong bound. Note that, say
$$\|T((a_\xi)_{\xi\in \F_{4,\lambda}})\|_{L^5}\le \|T((a_\xi)_{\xi\in \F_{4,\lambda}})\|_{L^\infty}$$
$$
\le N^2\|a_\xi\|_{l^\infty}\le N^2\|a_\xi\|_{l^5}.
$$
Restricted type interpolation now shows that for each $\epsilon$
$$\|T((a_\xi)_{\xi\in \F_{4,\lambda}})\|_{L^{4+\epsilon}}\lesssim_\epsilon N^{\frac17+\epsilon'}\|a_\xi\|_{l^{2+\epsilon''}},$$
where $\epsilon',\epsilon''\to 0$ as $\epsilon\to 0$.
This trivially implies that for each $\epsilon>0$
$$\|T((a_\xi)_{\xi\in \F_{4,\lambda}})\|_{L^{4}}\lesssim_\epsilon N^{\frac17+\epsilon}\|a_\xi\|_{l^{2}}.$$
We conclude that for $\alpha>0$
\begin{equation}
\label{hgfhtjyjyuu7ilk,gfd13}
|\{|F|>\alpha\}|\lesssim_\epsilon\alpha^{-4}N^{\frac47+\epsilon}.
\end{equation}
Next note that \eqref{bnew13fnguithtuirjioju} gives
\begin{equation}
\label{hgfhtjyjyuu7ilk,gfd12}
|\{|F|>\alpha\}|\lesssim \alpha^{-6}N^2, \;\;\alpha\gtrsim N^{3/4}
\end{equation}
Combining \eqref{hgfhtjyjyuu7ilk,gfd13} with \eqref{hgfhtjyjyuu7ilk,gfd12}, we get for $p>6$
$$\int_{\T^4}|F|^p=\int_{N^{3/4}\lesssim |F|\lesssim N^{1+\epsilon}}|F|^p+\int_{|F|\lesssim N^{3/4}}|F|^p\lesssim_{\epsilon}$$
$$N^{p-4+\epsilon}+N^{\frac34(p-4)}\int_{\T^4}|F|^4\lesssim_\epsilon N^{p-4+\epsilon}+N^{\frac47+\frac34(p-4)+\epsilon}.$$
It suffices now to note that this is bounded by $N^{p-4+\epsilon}$ when $p\ge \frac{44}{7}$.

When $n=5$ we will rely instead on the  sharp $L^4$ estimate that follows from \eqref{bnew13fnguithtuirjioju}
$$
|\{|F|>\alpha\}|\lesssim \alpha^{-4}N, \;\;\alpha\gtrsim N.
$$
Also \eqref{dbfhrgyfrkwpp-qp=-} and \eqref{e3e2} give
$$\E(\Lambda)\lesssim_\epsilon N^{4/3+\epsilon}|\Lambda|^2$$
for each $\Lambda\subset\F_{5,\lambda}$.

\section{Possible further improvements}
\label{sec7}

Our incidence theory approach relies on two ingredients. The main one is the hyperplane-point incidence theorem, which exploits the fact that quadruples contributing to the additive energy of the sphere concentrate on  hyperplanes. This theorem relies crucially on the topology of $\R^n$, as can be seen in the proof of the Cutting Lemma \ref{Cutting_Lemma}. The second ingredient
 is the fact that circles contain a negligible number of lattice points, which seems to be a rather weak use of the fact that our points lie on the sphere. It may be possible that by using finer properties about the distribution of lattice points on spheres, one might gain additional information about the relevant hyperplanes, and possibly further improve the estimates on the energy.

In light of the sharp subcritical estimate from \cite{Bo2}, one might wonder whether further progress is possible in the supercritical regime of Conjecture \ref{conj1} by methods that completely avoid number theory. We believe the answer is yes. It seems natural to conjecture that for $p\ge \frac{2(n+1)}{n-1}$
\begin{equation}
\label{Snew26}
\|\sum_kf_k\|_p\lesssim_\epsilon \delta^{-\frac{n-1}{4}+\frac{n+1}{2p}-\epsilon}(\sum_k\|f_k\|_p^2)^{1/2},
\end{equation}
for each partition of the unit sphere $S^{n-1}$ into $\delta^{1/2}-$ caps $C_k$, and each $\widehat{f_k}$ supported on a $\delta$ neighborhood of $C_k$. See (1.5) and (1.8) in \cite{GSS} for some partial results in this direction, in the more general (and difficult) case of cones. If this conjecture is indeed true, its proof would naturally  not involve any number theory. Moreover, using $\delta=N^{-2}$, it would imply via the dilation argument and the use of Dirac deltas as in \cite{Bo2} that
\begin{equation}
\label{Snew25}
\|\sum_{\xi\in \F_{n,\lambda}}a_\xi e(\xi\cdot x)\|_{L^p(\T^n)}\lesssim_\epsilon N^{\frac{n-1}{2}-\frac{n+1}{p}+\epsilon}\|a_\xi\|_{l^2(\F_{n,\lambda})},
\end{equation}
for each $a_\xi\in \C$, $\epsilon>0$ and each $p\ge \frac{2(n+1)}{n-1}$. On the other hand, if we assume \eqref{Snew25} for the critical index $p=\frac{2(n+1)}{n-1}$, and if we combine this
\eqref{bnew13fnguithtuirjioju} as in section \ref{sec6}, we further improve the range in Conjecture \ref{conj1} to $p\ge 6$ when $n=4$ and $p\ge 4$ when $n=5$.

We mention as a side remark that \eqref{Snew26} is expected to be true in the case of the (truncated) paraboloid. If indeed true,  this would in turn completely solve the discrete analog for the paraboloid considered in \cite{Bo3}.

This discrepancy between the sphere and the paraboloid in the discrete world is due to the non uniform distribution of lattice points on the sphere. It is likely that to detect these irregularities and get the full range in Conjecture \ref{conj1}, some involved number theory will be needed. One  step in this direction is made by our use of Siegel's mass formula, which produces sharp results for the energy of the whole sphere. Another possible avenue is described in the last section of \cite{BD}.

\end{document}